\documentclass[10pt,a4paper]{article}
\usepackage[utf8]{inputenc}
\usepackage[T1]{fontenc}
\usepackage{amsmath}
\usepackage{amsfonts}
\usepackage{amssymb,amsthm}
\usepackage{graphicx}
\usepackage{mathdots}
\usepackage{hyperref}
\usepackage[capitalise]{cleveref}
\usepackage{xstring}
\usepackage{footnote}
\def \C{\mathbb{C}}
\newcommand{\norm}[1]{\left\lVert#1\right\rVert}
\newcommand{\normp}[1]{\lvert\lvert\lvert#1\rvert\rvert\rvert}
\newcommand\scalemath[2]{\scalebox{#1}{\mbox{\ensuremath{\displaystyle #2}}}}
\newcommand{\rotanti}[1]{\rotatebox{90}{#1}}

\newcommand\inner[2]{\langle #1, #2 \rangle}
\newtheorem{theorem}{Theorem}[section]
\newtheorem{lemma}[theorem]{Lemma}
\newtheorem{remark}[theorem]{Remark}
\newtheorem{definition}[theorem]{Definition}
\newtheorem{corollary}[theorem]{Corollary}

\newtheorem{example}[theorem]{Example}

\numberwithin{equation}{section}
\numberwithin{table}{section}
\numberwithin{figure}{section}
\author{Biswajit Das\thanks{%
		Department of Mathematics,
		Indian Institute of Technology Guwahati,
		Guwahati 781039, Assam, India,
		(email:{biswajit.das@iitg.ac.in, shbora@iitg.ac.in}). The work of the first author is supported by the MHRD, Government of India.}
	~and
	Shreemayee Bora\footnotemark[1]}

\title{Nearest matrix polynomials with a specified elementary divisor.}
\date{}
\begin{document}
\maketitle

\begin{abstract}
	The problem of finding the distance from a given $n \times n$ matrix polynomial of degree $k$  to the set of matrix polynomials having the elementary divisor $(\lambda-\lambda_0)^j, \, j \geqslant r,$ for a fixed scalar $\lambda_0$ and $2 \leqslant r \leqslant kn$ is considered. It is established that polynomials that are not regular are arbitrarily close to a regular matrix polynomial with the desired elementary divisor. For regular matrix polynomials the problem is shown to be equivalent to finding minimal structure preserving perturbations such that a certain block Toeplitz matrix becomes suitably rank deficient. This is then used to characterize the distance via two different optimizations. The first one shows that if $\lambda_0$ is not already an eigenvalue of the matrix polynomial, then the problem is equivalent to computing a generalized notion of a structured singular value. The distance is computed via algorithms like BFGS and Matlab's globalsearch algorithm from the second optimization. Upper and lower bounds of the distance are also derived and numerical experiments are performed to compare them with the computed values of the distance.
\end{abstract}



\section{Introduction}

Given an matrix polynomial $P(\lambda) = \sum_{i=0}^k\lambda^iA_i$ of degree $k$ where $A_i, \, i = 0, \ldots, k$ are $n \times n$ real or complex matrices, this paper investigates the distance from $P(\lambda)$ to a nearest matrix polynomial with an elementary divisor $(\lambda - \lambda_0)^j, j \geqslant r,$ for a given $\lambda_0 \in \C$ and integer $r \geqslant 2.$ 

Although the problem is considered only for finite values of $\lambda_0,$ the analysis also covers the infinite case which is equivalent to the reversal polynomial defined by $\mathrm{rev}\, P(\lambda) := \sum_{i=0}^k \lambda^iA_{k-i}$ having an elementary divisor $\lambda^j, j \geqslant r.$ In particular in such cases the distance under consideration is important from the point of view of control theory for the following reasons. If $P(\lambda) = \lambda A_1 - A_0$ is a regular matrix pencil then there exist invertible matrices $E$ and $F$ such that $$EP(\lambda)F = \lambda\left[\begin{matrix} I_{p_1} & 0 \\ 0 & N_{p_2} \end{matrix}\right] + \left[\begin{matrix} J_{p_1} & 0 \\ 0 & I_{p_2} \end{matrix}\right], p_1 + p_2 = n,$$ where $J_{p_1}$ is a block diagonal matrix containing all the Jordan blocks associated with finite eigenvalues of $P(\lambda)$ and $N_{p_2}$ is a nilpotent matrix of size $p_2$ with nilpotent Jordan blocks on the diagonal. The block $N_{p_2}$ arises in the decomposition only if $P(\lambda)$ has an eigenvalue at $\infty.$ The matrix pencil on the right hand side of the above decomposition is called the Weierstrass canonical form of the pencil $P(\lambda).$ If $\infty$ is an eigenvalue, then the smallest positive integer $\nu$ such that $N_{p_2}^\nu = 0$ is called the index of the pencil. If $\nu > 1,$ then this is equivalent to the existence of a Jordan chain of length at least $2$ at $\infty$ for $P(\lambda),$  or equivalently an elementary divisor $\lambda^j, j \geqslant 2,$ for $\mathrm{rev}\, P(\lambda).$ In such a case the associated differential algebraic equation $A_1\dot{x}(t) = A_0x(t) + Bu(t),$ may not have any solution for certain choices of initial conditions unless the controller $u(t)$ is sufficiently smooth. In fact the larger the length of a Jordan chain at $\infty,$ the greater are the smoothness requirements on $u(t).$ In particular, for dynamical systems arising from matrix pencils as above to be stable or asymptotically stable, it is necessary that the matrix pencil has index at most one. Moreover, for the stability of such systems it is necessary that the purely imaginary eigenvalues of $P(\lambda)$ are not associated with Jordan chains of length $2$ or more. For more details see~\cite{ByeN93, Var95, DuLM13} and references therein.

It is well known that arbitrarily small perturbations to matrix pencils with $\lambda_0$ as an eigenvalue of algebraic multiplicity $r$ can result in a matrix pencil having an elementary divisor $(\lambda - \lambda_0)^r.$ In fact this result can also be extended to all matrix polynomials a proof of which is provided in \cref{zerocase}. Due to this fact, the distance problem under consideration is also equivalent to finding the distance to a nearest matrix polynomial with an eigenvalue at $\lambda_0$ with algebraic multiplicity at least $r.$ This problem has been considered in the literature in various forms. The distance to a nearest matrix polynomial with a prescribed multiple eigenvalue is considered in~\cite{PapP08} and bounds on the distance are obtained under certain conditions. In~\cite{Psa12} this work is extended to find the distance from a given matrix polynomial to a nearest matrix polynomial with a specified eigenvalue of algebraic multiplicity at least $r$ and a Jordan chain of length at most $k$ and an upper bound of the distance to a nearest matrix polynomial with a specified eigenvalue of algebraic multiplicity at least $r.$ The latter is done by constructing a perturbation to the given matrix polynomial which has the desired feature. However, the construction is possible under certain conditions. The results are extended to matrix polynomials in~\cite{KokPL18} where a similar construction is made  to find an upper bound on the distance to a nearest matrix polynomial with specified eigenvalues of desired multiplicities. The distance from an $n \times m$ matrix pencil $A + \lambda B$ with $n \geqslant m,$ to a nearest matrix pencil having specified eigenvalues such that the sum of their multiplicities is at least $r$ is considered in~\cite{KreMNT14}. Under the assumption that $\mathrm{rank} \, B \geqslant r,$ and only $A$ is perturbed, the distance is shown to be given by a certain singular value optimization under certain conditions. These ideas are extended in~\cite{KarM15} to find the same distance from a square matrix polynomial that has no infinite eigenvalues. Under certain conditions similar to those in~\cite{KreMNT14}, a singular value optimization is shown to be equal to the distance when only the constant coefficient of the matrix polynomial is perturbed. A lower bound is found for the general case when all coefficient matrices are perturbed. The techniques are further extended to find the same distance for more general nonlinear eigenvalue problems in~\cite{KarKM14}.

The analysis of the distance problem in this paper has several key features. Firstly the stated distance is considered for a square matrix polynomial that is either regular or singular and perturbations are considered on all the coefficient matrices of the polynomial. Note that with the exception of~\cite{KreMNT14} where a rectangular matrix pencil is considered, in all other works in the literature the matrix pencil or polynomial is assumed to be regular. However, \cite{KreMNT14} considers perturbations only to the constant coefficient matrix of the pencil. In fact by using elementary perturbation theoretic arguments it is shown in \cref{zerocase} that if the matrix polynomial $P(\lambda)$ is singular, then it is arbitrarily close to a regular matrix polynomial with an elementary divisor $(\lambda - \lambda_0)^j, j \geqslant r.$ This makes it possible to assume that the matrix polynomial $P(\lambda)$ is regular in the rest of the paper. A necessary and sufficient condition is obtained for $P(\lambda)$ to have $\lambda_0$ as an eigenvalue of algebraic multiplicity at least $r.$ Due to this it is possible to show that finding the stated distance is equivalent to finding a structure preserving perturbation such that the nullity of a certain block Toeplitz matrix is at least $r.$ This leads to a lower bound on the distance and allows for several characterizations of the distance in terms of optimization problems. Under the mild assumption that $\lambda_0$ is not an eigenvalue of $P(\lambda),$ for different choices of norms it is established that computing the distance from $P(\lambda)$ to a nearest matrix polynomial with an elementary divisor $(\lambda - \lambda_0)^j, j \geqslant r,$ is equivalent to computing a generalized version of a structured singular value or $\mu$-value. It is well known that the $\mu$-value computation is an NP-hard problem~\cite{BraYDM94}. Due to the form of the generalized $\mu$-value, these results are likely to throw light on the computational complexity of the distance problem. The characterization in terms of generalized $\mu$-values also yields a lower bound on the distance. Alternatively, the distance is characterized by another optimization problem which is computed via BFGS and Matlab's {\tt globalsearch} algorithm. This also results in an upper bound on the distance. A special case for which the solution of the distance problem has a closed form expression is also discussed. Finally, computed values of the distance via BFGS and Matlab's globalsearch are compared with upper and lower bounds.

\section{Preliminaries}\label{basic}

Standard notations are followed throughout the paper. The set of $n \times n$ complex matrices is denoted by $\mathbb{C}^{n \times n}.$ The $i$-th singular value of a matrix $A$ is denoted by $\sigma_i(A).$ Also the smallest singular value of $A$ is denoted by $\sigma_{\min}(A).$

Consider the matrix polynomial of degree $k$ of the form $P(\lambda)=\sum_{i=0}^k \lambda^i A_i$, $A_i\in \mathbb{C}^{n\times n}$.
There exist two $n \times n$ matrix polynomials $E(\lambda)$ and $F(\lambda)$ with nonzero determinants independent of $\lambda$ such that $P(\lambda)=E(\lambda)D(\lambda)F(\lambda)$, where $$D(\lambda)=\begin{bmatrix}d_1(\lambda)&&&&&\\&\ddots&&&&\\&&d_t(\lambda)&&&\\&&&0&&\\&&&&\ddots&\\&&&&&0\end{bmatrix}$$ is a diagonal matrix  with monic scalar polynomials $d_i(\lambda)$ such that $d_{i-1}(\lambda)$ divides $d_i(\lambda)$. This is called the Smith form of $P(\lambda)$. Important concepts associated with $P(\lambda)$ may be defined via its Smith form. The nonzero diagonal elements $d_1(\lambda),\dots,d_t(\lambda)$ are called the invariant polynomials of $P(\lambda)$. The number of such invariant polynomials is the normal rank of $P(\lambda)$. The polynomial $P(\lambda)$ is said to be regular if its normal rank is equal to its size $n.$ Else it is said to be a non-regular or singular matrix polynomial.

Each invariant polynomial may be written as a product of linear factors $$d_i(\lambda)=(\lambda-\lambda_{i1})^{c_{i1}}\dots(\lambda-\lambda_iq_i)^{c_{iq_i}}$$ where $\lambda_{i1},\dots,\lambda_{iq_i}$ are distinct complex numbers and $c_{i1},\dots,c_{iq_i}$ are positive integers. The factors $(\lambda-\lambda_{ij})^{c_{ij}}$  are called elementary divisors of $P(\lambda)$. Any $\lambda_0\in \mathbb{C}$ is a finite eigenvalue of $P(\lambda)$ if $(\lambda-\lambda_0)^c$ is a elementary divisor of $P(\lambda)$ for some positive integer $c$. The algebraic multiplicity of $\lambda_0$ as an eigenvalue of $P(\lambda)$ is the sum of all the powers of the term $(\lambda-\lambda_0)$ in all the invariant polynomials and the geometric multiplicity of $\lambda_0$ as an eigenvalue is the number of invariant polynomials which have $(\lambda-\lambda_0)^c$ as a factor.  Clearly, if the matrix polynomial $P(\lambda)$ is regular, then the eigenvalues of $P(\lambda)$ are the roots of $\det(P(\lambda))$ with algebraic multiplicity equal to the multiplicity of the root.

Having an elementary divisor $(\lambda - \lambda_0)^r$ is also equivalent to the existence of vectors $x_0,\dots,x_{r-1}\in\mathbb{C}^n,$ $x_0 \neq 0,$ satisfying the equations $$\sum_{i=0}^p\frac{1}{i!}P^i(\lambda_0)x_{p-i}=0, p=0,\dots,r-1$$ where $P^i(\lambda)$ denotes the $i$-th derivative of $P(\lambda)$ with respect to $\lambda$. The vectors $x_0, \ldots, x_{r-1}$ are said to form a Jordan chain of length $r$ of $P(\lambda)$ corresponding to $\lambda_0.$

Given any $n \times n$ matrix pencil $L(\lambda)=A-\lambda E$  there exist two $n\times n$ invertible matrices $P$ and $Q$ such that $P(A-\lambda E)Q$ is a block diagonal matrix pencil, the diagonal blocks being either $q \times q$ blocks of the form $\lambda I_q-J_q(\alpha)$ or $\lambda J_q(0)-I_q$, or $q \times (q+1)$ blocks of the form $\lambda G_q- F_q,$ or their transposes $\lambda G_q^T- F_q^T,$ where $$J_q(\alpha)=\begin{bmatrix}\alpha&1&&\\&\alpha&\ddots&\\&&\ddots&1\\&&&\alpha\end{bmatrix}, F_q=\begin{bmatrix}1&0&&\\&\ddots&\ddots&\\&&1&0\end{bmatrix}, G_q=\begin{bmatrix}0&1&&\\&\ddots&\ddots&\\&&0&1\end{bmatrix}$$ for some $\alpha\in \mathbb{C}$. The blocks $\lambda I_q-J_q(\alpha)$, $\lambda J_q(0)-I_q$, $\lambda G_q-F_q$ and $\lambda G_q^T-F_q^T$ correspond to a finite eigenvalue $\alpha$, the infinite eigenvalue, right singular blocks and left singular blocks respectively. This is called the Kronecker canonical form (KCF) of the pencil. For a matrix pencil, having a elementary divisor $(\lambda-\lambda_0)^r$ is equivalent to having a block $\lambda I_r-J_r(\lambda_0)$ in its KCF. Also clearly the algebraic multiplicity of $\lambda_0$ as an eigenvalue of $L(\lambda)$ is the sum of all the sizes of the blocks corresponding to $\lambda_0$ and its geometric multiplicity is the number of such blocks.

The normwise distance of $P(\lambda)$ to the set of all matrix polynomials having a elementary divisor $(\lambda - \lambda_0)^j$ where $j\geqslant r$
will be considered with respect to the following norms.
\begin{align*}
\scalemath{0.9}\delta_F(P,\lambda_0,r)=\inf\left\{\normp{\Delta P}_F|P+\Delta P\text{ has an elementary divisor } (\lambda-\lambda_0)^j, j\geqslant r \right\},\\
\delta_2(P,\lambda_0,r)=\inf\left\{\normp{\Delta P}_2|P+\Delta P\text{ has an elementary divisor } (\lambda-\lambda_0)^j, j\geqslant r \right\},
\end{align*}
where  $\normp{P}_F := \left(\sum_{i=0}^k\| A_i \|_F^2\right)^{1/2}$ and $\normp{P}_2 := \|[A_0 \cdots A_k]\|_2$,
$\| \cdot \|_F$ and $\| \cdot \|_2$ being the Frobenius and $2$-norms on matrices respectively. Also the matrix polynomial $\Delta P(\lambda)=\sum_{i=0}^k \lambda^i \Delta A_i$ is such that any of the coefficient matrix $\Delta A_i$ may be zero.

Due to the importance of the case $\lambda_0 = 0$ in practical applications and also because the results for this case involve expressions that are relatively simpler than the general case, in many instances initially the results for this special case are obtained and then extend for other choices of $\lambda_0.$ The following lemma will be useful for making these extensions.

\begin{lemma}\label{z_to_nz} Given any $n \times n$ matrix polynomial $Q(\lambda) = \sum_{i = 0}^k \lambda^iB_i,$ and $\lambda_0 \in \C,$
	$$\left[\begin{matrix} Q(\lambda_0) & Q'(\lambda_0) & \cdots & \frac{1}{p!}Q^p(\lambda_0) \end{matrix}\right] = \left[\begin{matrix} B_0 & \cdots & B_k
	\end{matrix}\right] M(\lambda_0;r)$$
	$M(\lambda_0;r)$ being a $(k+1)n \times (p+1)n$ matrix with $p =\min\{r,k\}$, given by
	$M(\lambda_0;r) = H(\lambda_0) \otimes I_n$ where $H(\lambda_0)$ is a $(k+1) \times (p+1)$ matrix with $(i,j)$ entry equal to $\left.\frac{1}{(j-1)!}\frac{d^{j-1}\lambda^{i-1}}{d\lambda^{j-1}}\right|_{\lambda=\lambda_0}.$
\end{lemma}

\begin{proof} The proof follows from the fact that for each $i = 1, \dots, p+1,$ $\frac{Q^{(i-1)}(\lambda_0)}{(i-1)!}$ is given by the product of $\left[\begin{matrix} B_0 & \cdots & B_k \end{matrix}\right]$ with $H_i(\lambda_0) \otimes I_n,$ $H_i(\lambda_0)$ being the $i$-th column of $H(\lambda_0).$
\end{proof}

\section{Polynomials for which the distance is zero}\label{zerocase}
Given a matrix polynomial $P(\lambda)$ it is interesting to identify cases when the distance under consideration is zero. One such situation is obviously the case that $k=1$ and $\lambda_0$ is an eigenvalue of $P(\lambda)$ of multiplicity at least $r.$ The main result of this section is a proof of the fact that this holds for all values of $k$ and also for singular matrix polynomials. The following theorem proves this for matrix pencils which is later generalized to matrix polynomials. For the sake of completeness, the case that the distance is zero if $\lambda_0$ is an eigenvalue of the pencil of multiplicity at least $r$ is also included in the statement of the theorem. Also note that although the theorem is proved with respect to the norm $\normp{ \cdot }_F,$ clearly it also holds for all other choices of norms.

\begin{theorem}\label{dist_sing}
	For a given  $n\times n$ matrix pencil $L(\lambda)=A-\lambda E$ and a positive integer $r \leqslant n,$ if
	\begin{enumerate}
		\item[(a)] $L(\lambda)$ is regular and algebraic multiplicity of $\lambda_0$ as an eigenvalue of $L(\lambda)$ is greater than or equal to $r$,  or
		\item[(b)] $L(\lambda)$ is singular,
	\end{enumerate}
	then it is arbitrarily close to a regular pencil having an elementary divisor $(\lambda-\lambda_0)^j$ where $j\geqslant r$.
\end{theorem}

\begin{proof}
	\item[(a)] The proof of this part is obvious owing to the structure of the Kronecker canonical form of a pencil having $\lambda_0$ as an eigenvalue of algebraic multiplicity atleast $r.$
	\item[(b)] Let $L(\lambda)$ be a singular pencil and $\epsilon > 0$ be arbitrarily chosen.
	Without loss of generality it may be assumed that $L(\lambda)$ is in Kronecker canonical form, i.e., $$L(\lambda) =\left[\begin{array}{c|c|c}R_f(\lambda)&&\\\hline &R_{inf}(\lambda)&\\\hline&&S(\lambda)\end{array}\right]$$ where $R_f(\lambda)$ and $R_{inf}(\lambda)$ represents the regular part of $L(\lambda)$ corresponding to finite and infinite eigenvalues respectively and $S(\lambda)$ represents the singular part. The idea of the proof is to construct a pencil $\Delta L(\lambda)$ such that $\normp{\Delta L}_F < \epsilon$ and $L(\lambda)+\Delta L(\lambda)$ is a regular matrix pencil with $0$ as an eigenvalue of algebraic multiplicity $n.$ Then by part (a), $L(\lambda)+\Delta L(\lambda)$ is arbitrarily close to having an elementary divisor $\lambda^n$ where clearly $n \geqslant r.$ The arguments are then extended to the case that $\lambda_0 \neq 0.$
	
	{\it Construction of $\Delta L(\lambda)$ for $\lambda_0 = 0:$} Initially it is assumed that all three types of blocks are present in $L(\lambda).$ Let the sizes of $R_f(\lambda)$, $R_{inf}(\lambda)$ and $S(\lambda)$ be $n_1\times n_1$, $n_2\times n_2$ and $n_3\times n_3$ respectively such that $n_1+n_2+n_3=n$. Also choose any $\tilde \epsilon \in (0, \epsilon ).$
	
	The block $R_f(\lambda)$ is bidiagonal with super-diagonal entries $0$ or $-1.$ Construct an $n_1 \times n_1$ pencil $\Delta R_f(\lambda)$ such that all the sub-diagonal entries are $\tilde \epsilon \lambda$ and all other entries are $0.$
	
	The block $R_{inf}(\lambda)$ is also bidiagonal with super-diagonal entries $\lambda$ or $0.$ Construct an $n_2 \times n_2$ pencil $\Delta R_{inf}(\lambda)$ by replacing the super-diagonal entries $\lambda$ and $0$ of  $R_{inf}(\lambda)$ by $0$ and $\tilde \epsilon \lambda$ respectively and setting all other entries to $0.$
	
	The singular part $S(\lambda)$ contains equal number of right and left singular blocks. Without loss of generality assume that right singular block and left singular blocks in  $S(\lambda)$ appear alternatively so that $S(\lambda)$ can be considered block diagonal with square diagonal blocks formed by placing one right and one left singular block next to each other. Each such diagonal block of $S(\lambda)$ has exactly one row and one column independent of $\lambda$ and all other rows and columns have exactly one entry as $\lambda.$ Assuming that there are $\mu$ blocks in $S(\lambda)$, suppose that $i_1\text{th},\dots,i_\mu\text{th}$ row and $j_1\text{th},\dots,j_\mu\text{th}$ columns of $S(\lambda)$ are independent of $\lambda.$ Construct an $n_3 \times n_3$ block diagonal pencil $\Delta S(\lambda)$ with square blocks on the diagonal of the same size as the diagonal blocks of $S(\lambda)$ such that the $(i_2,j_2)$th, $\dots$, $(i_\mu,j_\mu)$th entries are $\tilde \epsilon \lambda$ and all other entries are $0$.
	
	\noindent  Set $\Delta L(\lambda)=\left[\scalemath{.75}{
		\begin{array}{ccc|ccc|ccc}
		&&&\tilde \epsilon \lambda&&&&&\\
		&\Delta R_f(\lambda)&&&&&&&\\
		&&&&&&&&\\\hline
		&&&&&&&&\\
		&&&&\Delta R_{inf}(\lambda)&&&&\\
		&&&&&&&\tilde \epsilon \lambda&\\\hline
		&&&&&&&&\\
		&&\tilde \epsilon \lambda&&&&&\Delta S(\lambda)&\\
		&&&&&&&&
		\end{array}}\right],$ where $\tilde \epsilon \lambda$ has been placed in the $(1,n_1+1)$th, $(n_1+n_2,n_1+n_2+j_1)$th and $(n_1+n_2+i_1,n_1)$th positions of $\Delta L(\lambda).$ Choose $\tilde\epsilon$ small enough so that $\normp{\Delta L}_F < \epsilon.$ Now,
	$$L(\lambda)+\Delta L(\lambda)=\left[\scalemath{.85}{\begin{array}{ccc|ccc|ccc}&&&\tilde\epsilon \lambda&&&&&\\&\hat R_f(\lambda)&&&&&&&\\&&&&&&&&\\\hline  &&&&&&&&\\&&&&\hat R_{inf}(\lambda)&&&&\\&&&&&&&\tilde\epsilon \lambda&\\\hline&&&&&&&&\\&&\tilde\epsilon \lambda&&&&&\hat S(\lambda)&\\&&&&&&&&\end{array}}\right],$$ where
	$$\hat R_f(\lambda):=R_f(\lambda)+\Delta R_f(\lambda)=\begin{bmatrix}\lambda-\lambda_1&*&&\\\tilde\epsilon\lambda&\ddots&\ddots&\\&\ddots&\lambda-\lambda_{n_1-1}&*\\&&\tilde\epsilon\lambda&\lambda-\lambda_{n_1}\end{bmatrix},$$
	with  $*$ representing either $0$ or $-1$;
	$$\hat R_{inf}(\lambda)=R_{inf}(\lambda)+\Delta R_{inf}(\lambda)=\begin{bmatrix}-1&\star&&\\&\ddots&\ddots&\\&&-1&\star\\&&&-1\end{bmatrix},$$
	with $\star$ representing $\lambda$ or $\tilde\epsilon\lambda$ and $\hat S(\lambda)=S(\lambda)+\Delta S(\lambda).$ For $1 \leqslant \tilde i_1, \tilde i_2, \tilde j_1, \tilde j_2 \leqslant n,$ let $F[\tilde i_1,\tilde i_2; \tilde j_1, \tilde j_2](\lambda)$ denote the determinant of the submatrix of $L(\lambda) + \Delta L(\lambda)$ obtained by deleting rows $\tilde i_1, \tilde i_2$ and columns $\tilde j_1, \tilde j_2.$
	
	
	The determinant $\det (L(\lambda) + \Delta L(\lambda))$ is evaluated by first expanding along row $n_1 + n_2$ which has two nonzero entries, viz., $\tilde\epsilon \lambda$ at $n_1 + n_2 + j_1$ position and  $-1$ at $n_1 + n_2$ position and then expanding along column $n_1$ which has atmost three nonzero entries, viz., $\tilde\epsilon \lambda$ at $n_1 + n_2 + i_1$ position, $\lambda - \lambda_{n_1}$ at $n_1$ position and $-1$ (if $* = -1$) at $n_1 - 1$ position. Setting $\tilde i_1 = n_1 + n_2, \tilde j_1 = n_1 + n_2 + j_1, \tilde i_2 = n_1 + n_2 + i_1, \tilde j_2 = \tilde i_3 = n_1,$ $\tilde j_3 = \tilde i_1,$ and $\tilde i_4 = n_1 - 1,$
	
	\begin{eqnarray*}
		\det (L(\lambda)+\Delta L(\lambda)) & = & \pm (\tilde\epsilon \lambda)^2 F[\tilde i_1, \tilde i_2; \tilde j_1, \tilde j_2](\lambda) \pm \tilde\epsilon \lambda(\lambda - \lambda_{n_1})  F[\tilde i_1, \tilde i_3; \tilde j_1, \tilde j_2](\lambda) \\
		& \pm & \tilde\epsilon c \lambda  F[\tilde i_1, \tilde i_4; \tilde j_1, \tilde i_3](\lambda) \pm \tilde\epsilon \lambda  F[\tilde i_1, \tilde i_2; \tilde j_3, \tilde j_2](\lambda)\\
		& \pm & (\lambda - \lambda_{n_1})  F[\tilde i_1, \tilde i_3; \tilde j_3, \tilde j_2](\lambda) \pm  cF[\tilde i_1, \tilde i_4; \tilde j_3, \tilde j_2](\lambda)
	\end{eqnarray*}
	where $c = 0$ or $-1.$
	
	As $\hat S (\lambda)$ and its submatrices obtained by removing row $i_1$ or column $j_1$ are all singular pencils, the last five terms in the right hand side of the above expression  are zero. Now $F[\tilde i_1, \tilde i_2; \tilde j_1, \tilde j_2](\lambda)$ is the product of the determinants of two matrices, viz., $\hat S(\lambda)$ with rows $i_1$ and $j_1$ removed and the submatrix of
	$$\hat R(\lambda) := \left[\begin{array}{ccc|ccc}&&&\tilde\epsilon\lambda&&\\&\hat R_f(\lambda)&&&&\\&&&&&\\\hline&&&&&\\&&&&\hat R_{inf}(\lambda)&\\&&&&&
	\end{array}\right]$$
	obtained by removing column $n_1$ and the last row. Due to the manner of constructing $\hat S(\lambda),$ the determinant of $\hat S(\lambda)$ with rows $i_1$ and $j_1$ removed is given by $\pm \tilde\epsilon^{p_1}\lambda^{n_3 -1}$ for some positive integer $p_1.$ Let $G[\tilde i_1, \tilde i_2 ; \tilde j_1, \tilde j_2](\lambda)$ denote the determinant of $\hat R(\lambda)$
	with rows $\tilde i_1, \tilde i_2$ and columns $\tilde j_1$ and $\tilde j_2$ removed. Then the determinant of submatrix of $\hat R(\lambda)$ with the last row and column $n_1$ removed is given by $$\pm (\tilde\epsilon\lambda) G[n_1 + n_2, 1; n_1, n_1+1](\lambda) \pm G[n_1 + n_2, n_1 + 1;n_1, n_1+1](\lambda)$$ when expanded along column $n_1 + 1.$ But $G[n_1 + n_2, n_1 + 1;n_1, n_1+1](\lambda) = 0$ as the submatrix obtained by eliminating the indicated rows of $\hat R(\lambda)$ is singular.
	Therefore,
	\begin{equation}\label{eqone} \det (L(\lambda)+\Delta L(\lambda))  =  \pm (\tilde\epsilon \lambda)^2\left(\tilde\epsilon^{p_1}\lambda^{n_3-1}\right)\left(\tilde\epsilon \lambda G[n_1 + n_2, 1; n_1, n_1+1](\lambda)\right),
	\end{equation}
	where
	\begin{eqnarray}
	G[n_1 + n_2, 1; n_1, n_1+1](\lambda) & = & \mathrm{det} \left(\left[\begin{array}{cccc|cccc}\tilde\epsilon\lambda&\lambda-\lambda_2&*&&&&&\\
	&\ddots&\ddots&*&&&&\\
	&&\tilde\epsilon\lambda&\lambda-\lambda_{n_1-1}&&&&\\
	&&&\tilde\epsilon\lambda&&&&\\
	\hline&&&&\star&&&\\
	&&&&-1&\ddots&&\\
	&&&&&\ddots&\star&\\
	&&&&&&-1&\star\end{array}\right]\right) \nonumber \\
	&=& \tilde\epsilon^{p_2}\lambda^{n_1+n_2-2}, \label{eqtwo}
	\end{eqnarray}
	for some positive integer $p_2.$ From \eqref{eqone} and \eqref{eqtwo} $$\det (L(\lambda)+\Delta L(\lambda)) = \pm\tilde\epsilon^{p_1 + p_2 + 3}\lambda^{n_1 + n_2 + n_3} = \pm \tilde\epsilon^p\lambda^n,$$ for $p = p_1 + p_2 + 3.$ This establishes that $L(\lambda) + \Delta L(\lambda)$ is a regular matrix pencil with $0$ an eigenvalue of algebraic multiplicity $n.$
	
	In those cases where all three types of blocks $R_f(\lambda),$ $R_{inf}(\lambda)$ and $S(\lambda)$ are not present in $L(\lambda),$ the strategy for forming $\Delta L(\lambda)$ are as follows.
	\begin{itemize}
		\item If the block $R_f(\lambda)$ is not present then
		$\Delta L(\lambda)=\left[\scalemath{.75}{
			\begin{array}{ccc|ccc}
			&&&&&\\
			&\Delta R_{inf}(\lambda)&&&&\\
			&&&&\tilde\epsilon \lambda&\\\hline
			&&&&&\\
			\tilde\epsilon \lambda&&&&\Delta S(\lambda)&\\
			&&&&&
			\end{array}}\right]$, where the $(n_2,n_2+j_1)$ and $(n_2+i_1,1)$ entries of $\Delta \hat L(\lambda)$ are $\tilde\epsilon \lambda$ and the construction of $\Delta R_{inf}(\lambda)$ and $\Delta S(\lambda)$ are the same as above.
		\item If the block $R_{inf}(\lambda)$ is not present then
		$\Delta L(\lambda)=\left[\scalemath{.75}{
			\begin{array}{ccc|ccc}
			&&&&\tilde\epsilon \lambda&\\
			&\Delta R_f(\lambda)&&&&\\
			&&&&&\\\hline
			&&&&&\\
			&&\tilde\epsilon \lambda&&\Delta S(\lambda)&\\
			&&&&&
			\end{array}}\right]$, where the $(n_1+i_1,n_1)$ and $(1,n_1+j_1)$ entries of $\Delta \hat L(\lambda)$ are $\tilde\epsilon \lambda$ and the construction of $\Delta R_f(\lambda)$ and $\Delta S(\lambda)$ are the same as above.
		\item If only the singular blocks occur in $L(\lambda)$ then we construct $\Delta L(\lambda)$ as a block diagonal matrix with blocks of the same size as $S(\lambda)$ such that the $(i_1,j_1)$, $\dots$, $(i_\mu,j_\mu)$ entries are $\tilde\epsilon \lambda$ and all other entries are $0$.
	\end{itemize}
	In each case the above arguments may be extended to show that by choosing $\tilde{\epsilon}$ small enough, $\normp{\Delta L}_F < \epsilon$ and $L(\lambda)+\Delta L(\lambda)$ is a regular pencil whose determinant is a scalar multiple of $\lambda^n.$
	
	Now suppose $\lambda_0 \neq 0.$ The pencil $L(\lambda)$ may be written in the form $$L(\lambda) = A - \lambda_0E - (\lambda - \lambda_0)E.$$ Setting $\hat{A} = A - \lambda_0 E$ and $\hat{\lambda} = \lambda - \lambda_0$ and arguing as above, for a given $\epsilon > 0$ there exists $\Delta E \in \C^{n \times n}$ such $\hat{A} - \hat{\lambda} (E + \Delta E)$ is a regular pencil with $\mathrm{det} (\hat{A} - \hat{\lambda} (E + \Delta E)) = \pm\tilde\epsilon^p\hat\lambda^n$ for some $\tilde{\epsilon} \in (0, \epsilon)$ such that $\norm{\Delta E}_F < \epsilon.$ This implies that $$\mathrm{det}((A + \lambda_0 \Delta E) - \lambda(E + \Delta E)) = \pm\tilde\epsilon^p(\lambda-\lambda_0)^n.$$ Therefore for $\Delta L(\lambda) := (\lambda_0 \Delta E) - \lambda \Delta E,$ $\lambda_0$ is an eigenvalue of $(L + \Delta L)(\lambda)$ of algebraic multiplicity $n \geqslant r.$ The proof now follows from the fact that $\tilde{\epsilon}$ may be chosen small enough so that $\normp{\Delta L}_F = \sqrt{|\lambda_0|^2 + 1}\| \Delta E\|_F <  \epsilon.$ Hence the proof.
\end{proof}

The above result may be extended to matrix polynomials by considering the first companion linearization
$$C_1(\lambda)=\lambda \begin{bmatrix}A_k & & &\\ &I_n& & \\ & & \ddots &\\ & & & I_n\end{bmatrix}+\begin{bmatrix}A_{k-1} &A_{k-2} &\cdots &A_0\\ -I_n& & &0\\ & \ddots&  &\\ & & -I_n& 0\end{bmatrix}$$ of $P(\lambda) = \sum_{i = 0}^k \lambda^iA_i.$ It is an example of a block Kronecker linearization as introduced in~\cite{DopLPV18} where it was shown that if $L(\lambda)$ is a block Kronecker linearization of $P(\lambda)$ and $\Delta L(\lambda)$ is a pencil of the same size as $L(\lambda)$ with $\normp{\Delta L}_F<\epsilon,$ for some sufficiently small $\epsilon > 0,$ then $L(\lambda)+\Delta L(\lambda)$ is a strong linearization of $P(\lambda)+\Delta P(\lambda)$ such that $\normp{\Delta P}_F<C\epsilon$ for some positive constant $C.$ Due to this result the following theorem is an immediate consequence of \cref{dist_sing}.

\begin{theorem}\label{dist_is_zero}
	For a given $n \times n$ matrix polynomial $P(\lambda)$ of degree $k,$ and a positive integer $r \leqslant kn,$ if
	\begin{itemize}
		\item[(a)] $P(\lambda)$ is regular and $\lambda_0$ is an eigenvalue of algebraic multiplicity greater than or equal to $r,$  or
		\item[(b)] $P(\lambda)$ is singular,
	\end{itemize}
	then $P(\lambda)$ is arbitrarily close to a regular matrix polynomial having an elementary divisor $(\lambda-\lambda_0)^j$ where $j\geqslant r$.
\end{theorem}

In fact, by using the arguments in the proof of \cref{dist_sing}, it is clear that any $n \times n$ singular matrix polynomial $P(\lambda)$ of degree $k$ is arbitrarily close to a regular matrix polynomial having an elementary divisor $(\lambda-\lambda_0)^{kn}.$ In view of the above theorem, it is now possible to assume without loss of generality the distance $\delta_s(P,\lambda_0, r)$ for $s=2$ of $F$ are being computed for a regular matrix polynomial $P(\lambda)$ which does not have $\lambda_0$ as an eigenvalue of algebraic multiplicity $r.$ This also has the effect of removing the uncertainty that was earlier associated with the situation that perturbations being made to the matrix polynomial for the desired objectives could result in a singular matrix polynomial.

\section{A characterization via block Toeplitz matrices}
One of the aims of this work is to show that for appropriate choices of norms computing the distance to a matrix polynomial with an elementary divisor $(\lambda-\lambda_0)^j, j\geqslant r,$ is equivalent to finding a structured singular value or generalized $\mu$-value. The next result is an important step in this direction. Since the expression for the optimization is more aesthetic if $r$ is replaced by $r+1,$ in the rest of the paper the distance is considered in the form $\delta_s(P,\lambda_0,r+1),$ where $s = 2$ or $F.$ The following definition will be frequently used.

For any  $\gamma = \left[\begin{array}{cccc} \gamma_1 & \gamma_2 & \cdots & \gamma_r \end{array}\right] \in \Gamma$ given by \eqref{gamma} and $\alpha \in \C,$ let $T_\gamma(Q,\alpha)$ be a function from the set of all $n \times n$ matrix polynomial $Q(\lambda) = \sum_{i=0}^k \lambda^iA_i,$ to the set of $(r+1)n \times (r+1)n$ matrices defined by
\begin{equation}\label{tgamma}
T_\gamma(Q,\alpha):=\begin{bmatrix}
Q(\alpha)&&&&\\
\gamma_1 Q^{\prime}(\alpha)&Q(\alpha)&&&\\
\gamma_1 \gamma_2 \frac{Q^{\prime\prime}(\alpha)}{2!}&\gamma_2 Q^{\prime}(\alpha)&Q(\alpha)&&\\
\vdots&\vdots& \ddots&\ddots&\\
\displaystyle{\prod_{i=1}^r}\gamma_i \frac{Q^r(\alpha)}{r!} &\displaystyle{\prod_{i=2}^r}\gamma_i\frac{Q^{r-1}(\alpha)}{(r-1)!}  &\cdots & \gamma_rQ^{\prime}(\alpha) &Q(\alpha)
\end{bmatrix}.
\end{equation}

\begin{theorem}\label{sigma_result}
	A scalar $\lambda_0\in\mathbb{C}$ is a eigenvalue of a $n\times n$ matrix polynomial $P(\lambda)$ of algebraic multiplicity at least $r+1$ if and only if  the rank of $T_\gamma(P,\lambda_0)$ as defined by \eqref{tgamma} is at most $(r+1)(n-1).$
\end{theorem}

\begin{proof}
	Let $\lambda_0$ be an eigenvalue of $P(\lambda)$ of multiplicity $r+1$. Suppose there are $p$ Jordan chains $\left\{x_{11},x_{12}, \ldots, x_{1k_1}\right\}$, $\left\{x_{21},x_{22},\dots, x_{2k_2}\right\}$, $\dots$, $\left\{ x_{p1}, x_{p2}, \dots, x_{pk_p}\right\}$ of $P(\lambda)$ corresponding to $\lambda_0$ where $x_{ij} \in \C^n$ for $i = 1, \ldots, p,$ $j = 1, \ldots, k_i,$ satisfying $\sum_{i = 1}^p k_i \geqslant r+1.$ The $i^{\mathrm{th}}$ Jordan chain satisfies the following equations for $i = 1, \dots, p.$
	\begin{align*}
		&P(\lambda_0)x_{i1}=0\\
		&P(\lambda_0)x_{i2}+P^\prime(\lambda_0)x_{i1}=0\\
		&\vdots\\
		&P(\lambda_0)x_{ik_i}+P^\prime(\lambda_0)x_{i(k_i-1)}+\frac{P^{\prime\prime}(\lambda_0)}{2!}x_{i(k_i-2)}+\dots+\frac{P^{k_i-1}(\lambda_0)}{(k_i-1)!}x_{i1}=0.
	\end{align*}
	It may be assumed that $\{x_{11},x_{21},\dots,x_{p1}\}$ is a linear independent set. Clearly the $i^{\mathrm{th}}$ Jordan chain contributes the following $k_i$ vectors
	 $$\begin{bmatrix}0\\\vdots\\0\\ \frac{x_{i1}}{\gamma_{(r-k_i+2)}\cdots\gamma_r}\\\vdots \\ \frac{x_{i(k_i-3)}}{\gamma_{r-2}\gamma_{r-1}\gamma_r}\\\frac{x_{i(k_i-2)}}{\gamma_{r-1}\gamma_r}\\\frac{x_{i(k_i-1)}}{\gamma_r}\\x_{ik_i}\end{bmatrix},
	\begin{bmatrix}0\\0\\\vdots\\0\\ \frac{x_{i1}}{\gamma_{(r-k_i+3)}\dots\gamma_r}\\ \vdots \\\frac{x_{i(k_i-3)}}{\gamma_{r-1}\gamma_r}\\\frac{x_{i(k_i-2)}}{\gamma_r}\\x_{i(k_i-1)}\end{bmatrix},
	\begin{bmatrix}0\\0\\0\\\vdots\\0\\ \frac{x_{i1}}{\gamma_{(r-k_1+4)}\dots\gamma_r}\\ \vdots \\\frac{x_{i(k_i-3)}}{\gamma_r}\\x_{i(k_i-2)}\end{bmatrix},
	\dots, \begin{bmatrix}0\\0\\0\\0\\0\\0\\\vdots\\0\\x_{i1}\end{bmatrix}$$
	of length$(r+1)n$ to the null space $N(T_\gamma(P, \lambda_0))$ of $T_\gamma(P,\lambda_0)$ for $i=1,\dots,p$. All the above vectors are linearly independent as $\{x_{11},x_{21},\dots,x_{p1}\}$ are linearly independent. Hence the nullity of $T_\gamma(P,\lambda_0)$ is at least $r+1.$
	
	Conversely suppose that $\mathrm{rank}(T_\gamma(P, \lambda_0)) \leqslant (r+1)(n-r)$ so that the nullity of $T_\gamma(P, \lambda_0)$ is at least $r+1.$ Let $\{x_1,x_2,\dots ,x_{r+1}\}$ be a linearly independent ordered list in $N(T_\gamma(P,\lambda_0)),$ where $$x_j=\begin{bmatrix}x_{r+1,j}^T & \cdots & x_{i,j}^T & \cdots & x_{2,j}^T & x_{1,j}^T\end{bmatrix}^T$$ with $x_{i,j}\in\mathbb{C}^{n}$ for $i,j\in\{1,\dots,r+1\}.$  If $x_{r+1,j}\neq 0$ for some $j,$ then $x_j$ will be a Jordan chain of $P(\lambda)$ of length $r+1$ corresponding to $\lambda_0$ and the proof follows. So assume without loss of generality that for each $j = 1, \dots, r+1,$ there exists $t_j, 0 < t_j < r+1$ such that $x_{ij} = 0$ for all $i = t_j+1, \dots, r+1.$ Let $t = \max_{1 \leqslant j \leqslant r+1} t_j$ and $p = t - \min_{1 \leqslant j \leqslant r+1} t_j.$
	By reordering the list if necessary, it may be assumed that the first $k_1 + \cdots + k_s $ vectors of the list satisfy $x_{ij} = 0$ for all $s = 1, \dots, p+1$ and $i = t-s+2, \dots, r+1,$   so that $k_1 + \cdots k_{p+1} = r+1.$ Note that $k_j$ may be zero for some or all $j = 2, \dots, p+1.$ Consider $X = \left[\begin{array}{cccc} x_1 & x_2 & \cdots & x_{r+1} \end{array}\right].$ Then in fact,
	$$X=\scalemath{0.85}{\begin{bmatrix}
		0&\cdots&0&0&\cdots&0&\cdots&0&\cdots&0\\
		\vdots&\vdots&\vdots&\vdots&\vdots&\vdots&\vdots&\vdots&\vdots&\vdots\\
		0&\cdots&0&0&\cdots&0&\cdots&0&\cdots&0\\
		x_{t,1}&\cdots&x_{t,k_1}&0&\cdots&0&&0&\cdots&0\\
		\cdots&\cdots&\cdots&x_{t-1,k_1+1}&\cdots&x_{t-1,k_1+k_2}&\cdots&\cdots&\cdots&\cdots\\
		\vdots&\vdots&\vdots&\vdots&\vdots&\vdots&\vdots&0&\cdots&0\\
		\cdots&\cdots&\cdots&\cdots&\cdots&\cdots&\cdots&x_{t-p,r+2-k_{p+1}}&\cdots&x_{t-p,r+1}\\
		\vdots&\vdots&\vdots&\vdots&\vdots&\vdots&\vdots&\vdots&\vdots&\vdots\\
		x_{1,1}&\cdots&x_{1,k_1}&x_{1,k_1+1}&\cdots&x_{1,k_1+k_2}&\cdots&x_{1,r+2-k_{p+1}}&\cdots&x_{1,r+1}
		\end{bmatrix}}.$$
	It is possible that some of the vectors $x_{t-s+1, 1+ \sum_{j=1}^sk_j}, \ldots, x_{t-s+1, \sum_{j=1}^{s+1}k_j}$ in the consecutive columns $1+ \sum_{j=1}^{s-1}k_j$ to $\sum_{j=1}^sk_j$ of $X$ can be made zero for each $s = 1, \dots, p+1,$ via elementary column operations on $X$ that affect only those columns. The columns of the transformed $X$ will also be a linearly independent list in $N(T_\gamma(P,\lambda_0)).$ Assume without loss of generality that $X$ has been formed  after such transformations have already been made and the submatrices $$\left[\begin{array}{ccc} x_{t-s+1, 1+ \sum_{j=1}^{s-1}k_j} & \cdots & x_{t-s+1, \sum_{j=1}^{s}k_j}\end{array}\right]$$ of $X$ have full column rank for each  $s = 1, \dots, p+1.$
	
	Clearly the columns of $X$ are linearly independent and belong to $N(T_\gamma(P,\lambda_0)).$ Also the first $k_1$ columns give rise to Jordan chains of $P(\lambda)$ corresponding to $\lambda_0$ of length $t,$ the next $k_2$ columns give rise to Jordan chains of of length $t-1$ and so on, with the last $k_{p+1}$ columns forming Jordan chains of length $t-p.$ Due to the structure of $X,$ $\beta_s=\{x_{t-s+1,k_1+\dots+k_{s-1}+1},\dots,x_{t-s+1,k_1+\dots+k_{s}}\},$ $s=1,1,\dots,p+1,$ are ordered lists of linearly independent vectors. Consider the list $$\beta:=\beta_{p+1}=\{x_{t-p,k_1+\dots+k_p+1},\dots,x_{t-p,k_1+\dots+k_{p+1}}\}.$$ If the first vector of the list $\beta_p,$ does not belong to $\mathrm{span}(\beta),$ it is included in $\beta$. If it belongs to $\mathrm{span}(\beta)$ then it can be uniquely represented by a linear combination of the vectors of $\beta$ and at least one of the scalar coefficients in the representation is nonzero. Replace one of the vectors from $\beta$ whose associated coefficient in the linear combination is nonzero by the first vector of $\beta_p.$ Now consider the second vector of the list $\beta_p.$ If it does not belong to the span of the updated $\beta,$ then it is included in $\beta$. Otherwise it is a linear combination of the vectors of $\beta$ with at least one of the scalar coefficients in the linear combination being non zero. As $\beta_p$ is a linearly independent list, a vector associated with such a non zero scalar in the linear combination can be chosen from $\beta_{p+1}.$ The set $\beta$ is further updated by replacing this vector by the second vector from $\beta_{p}.$ This process is continued for the rest of the vectors in $\beta_{p}$ as well as those of $\beta_{p-1}, \dots, \beta_1.$ The final $\beta$ clearly forms a linearly independent list of eigenvectors of $P(\lambda)$ corresponding to $\lambda_0.$ Moreover the sums of the lengths of the Jordan chains associated with these eigenvectors is at least  $tk_1 + \sum_{s = 1}^p(t-s)(k_{s+1} - k_s).$ But
	$$tk_1 + \sum_{s = 1}^p(t-s)(k_{s+1} - k_s) = \sum_{s = 1}^pk_s + (t-p)k_{p+1} \geqslant  \sum_{s=1}^{p+1}k_s = r+1.$$
	Hence $\lambda_0$ is an eigenvalue of $P(\lambda)$ of algebraic multiplicity at least $r+1$ and the proof follows.
\end{proof}

\begin{remark} \cref{sigma_result} is established in \cite{PapP08} for the particular case that $P(\lambda)$ has an eigenvalue of multiplicity $2.$ Under the assumption that the leading coefficient matrix is full rank, another characterization of a matrix polynomial $P(\lambda)$ having a specified eigenvalue of multiplicity $r$ is obtained in~\cite{KarM15} via a different block Toeplitz matrix that involves $r(r+1)/2$ parameters.
\end{remark}

In view of part (a) of \cref{dist_is_zero}, the following corollary of \cref{sigma_result} is immediate.

\begin{corollary}\label{cor_sigma_result} Given any $n \times n$ matrix polynomial $P(\lambda)$ consider the collection $\mathcal{S}(P,\lambda_0)$ of all $n \times n$ matrix polynomials
	$(\Delta P)(\lambda) := \sum_{i = 0}^k \lambda^i \Delta A_i$ such that the block Toeplitz matrices $T_{\gamma}(P + \Delta P, \lambda_0)$ as defined in~\eqref{tgamma} with $\gamma = [1 \cdots 1] ,$
	have rank at most $(r+1)(n-1).$ For any choice of norm $\normp{\cdot},$ the distance to a nearest matrix polynomial
	with an elementary divisor $(\lambda - \lambda_0)^j, j \geqslant r+1,$ is given by $\inf\{\normp{\Delta P}: \Delta P(\lambda) \in \mathcal{S}(P,\lambda_0)\}.$
\end{corollary}

\section{The distance as the reciprocal of a generalized $\mu$ value}

\Cref{cor_sigma_result} implies that for any given choice of norm, finding the distance from $P(\lambda)$ to a nearest matrix polynomial with an elementary divisor $\lambda^{r+1}$ is equivalent to finding the smallest structure preserving perturbation to the block Toeplitz matrix
$$\begin{bmatrix}
P(0)&&&&\\
P^\prime(0)&P(0)&&&\\
\frac{1}{2!}P^{\prime\prime}(0) &  P^\prime(0) & P(0)&&\\
\vdots&\vdots& \ddots&\ddots&\\
\frac{1}{r!}P^r(0) & \frac{1}{(r-1)!}P^{r-1}(0)  &\cdots & P^\prime(0) & P(0)
\end{bmatrix}$$
so that the rank of the perturbed matrix is at most $(r+1)(n-1).$ This fact is used in this section to show that if $\lambda_0 \in \C$ is not already an eigenvalue of $P(\lambda),$ then computing the distance from $P(\lambda)$ to a nearest matrix polynomial with the desired elementary divisor $(\lambda - \lambda_0)^j, j \geqslant r+1,$ with respect to the norms $\normp{P}_2$ and $\normp{P}_F$ is the reciprocal of a generalized notion of a $\mu$-value. To this end, the definition of a perturbation class and a structured singular value, which is also referred to in the literature as a $\mu$-value, are introduced.

A perturbation class $S$ is a nonempty closed subset of $\mathbb{C}^{p\times q}$ such that if $\Delta\in S$ then $t\Delta \in S$ for $0\leqslant t\leqslant 1.$

\begin{definition}[$\mu$-value]~\cite{PacD93, Kar03}
	Let $S \subset \mathbb{C}^{p\times q}$ be a perturbation class and let $\|.\|$ be a norm on $\mathbb{C}^{p\times q}$.
	The $\mu$-value of $M\in \mathbb{C}^{q\times p}$ with respect to $S$ and $\|.\|$ is
	\begin{equation}
	\mu_{S,\norm{.}}(M):=(\inf\{\norm{\Delta}:\Delta\in S, 1\in\sigma(\Delta M)\})^{-1}.
	\end{equation}
	If there is no such $\Delta\in S$, then $\mu_{S,\norm{.}}(M)=0$.
\end{definition}

The generalized $\mu$-value is now defined as follows.

\begin{definition}[Generalized $\mu$-value]
	Let $S \subset \mathbb{C}^{p\times q}$ be a perturbation class and let $\norm{.}$ be a norm on $\mathbb{C}^{p\times q}$.
	The generalized $\mu$-value of $M\in \mathbb{C}^{q\times p}$ with respect to $S$ and $\|.\|$ is defined as
	\begin{equation}
	\mu_{S,\norm{.}}^{r}(M):=(\inf\{\norm{\Delta}:\Delta\in S, \mathrm{rank}(I-\Delta M)\leqslant p-r\})^{-1}.
	\end{equation}
	If there is no such $\Delta\in S$, then $\mu_{S,\norm{.}}^r(M)=0$.
\end{definition}

Let \begin{equation}\label{blocktoep} T(Q,\lambda_0) := T_\gamma(Q,\lambda_0) \text{ with } \gamma = [1 \cdots 1].\end{equation} The following lemma provides an useful factorization of $T(Q,\lambda_0).$ The proof of the lemma follows from direct multiplication of the stated factors and is therefore skipped.

\begin{lemma}\label{fact_T_1}
	For a given positive integer $r$ and an $n\times n$ matrix polynomial $Q(\lambda)$ of degree $k$,
	$$T(Q,\lambda_0)=\left(I_{r+1}\otimes\begin{bmatrix}Q(\lambda_0)&Q'(\lambda_0)&\cdots&\frac{Q^{\min\{r,k\}}(\lambda_0)}{\min\{r,k\}!}\end{bmatrix}
	\right)E$$
	where $$E=\begin{cases}
	\begin{bmatrix}E_1^T&E_2^T&\cdots&E_{r+1}^T\end{bmatrix}^T\otimes I_n& \text{ if } r\leqslant k\\
	\begin{bmatrix}\tilde E_1^T&\tilde E_2^T&\cdots&\tilde E_{r+1}^T\end{bmatrix}^T\otimes I_n& \text{ if } r> k
	\end{cases}
	$$
	such that $E_i,$ $i=1,\dots,r,r+1$ are the $(r+1)\times (r+1)$ matrices,
	$$E_1=\begin{bmatrix}1&&&\\&&&\\&&&\\&&&\end{bmatrix}, E_2=\begin{bmatrix}&1&&\\1&&&\\&&&\\ &&&\end{bmatrix}, \dots, E_r=\scalemath{.85}{\begin{bmatrix}&&1&\\&\iddots&&\\1&&&\\&&& \end{bmatrix}},
	E_{r+1} =\scalemath{0.85}{\begin{bmatrix}&&&1\\&&1&\\&\iddots&&\\1&&&\end{bmatrix}}$$ and $\tilde E_i\in\mathbb{C}^{(k+1)\times (r+1)}$ are the first $k+1$ rows of $E_i$.
\end{lemma}

The next theorem is the main result of this section.

\begin{theorem}\label{dist_as_mv} Let $P(\lambda)=\sum_{i=0}^{k}\lambda^i A_i$ be an $n \times n$ matrix polynomial of degree $k,$ and $1 \leqslant r < kn.$
	For $\Delta A_i \in \C^{n \times n}, i = 0, \dots, k,$ let $S_1$ be the perturbation class of all perturbations of the type $\left(I_{r+1}\otimes\begin{bmatrix}\Delta A_0&\cdots&\Delta A_k\end{bmatrix}\right)$ and $S_2$ be the perturbation class of all perturbations of the type $\left(I_{r+1}\otimes\begin{bmatrix}\Delta A_0&\cdots&\Delta A_{\min\{r,k\}}\end{bmatrix}
	\right).$
	For any $\lambda_0 \in \C$ which is not an eigenvalue of $P(\lambda),$ let $T(P,\lambda_0)$ be defined by~\eqref{blocktoep} and $E$ and $M(\lambda_0;r)$ be as given in
	\cref{fact_T_1} and \cref{z_to_nz} respectively. Then,
	$$\delta_2(P,\lambda_0,r+1)= \left\{\begin{array}{ll} \left[\mu_{S_1,\norm{.}_{2}}^{r+1}\left((I_{r+1}\otimes M(\lambda_0;r))E~(T(P,\lambda_0))^{-1}\right)\right]^{-1} & \text{ if } \lambda_0 \neq 0, \\
	\left[\mu_{S_2,\norm{.}_{2}}^{r+1}\left(E~(T(P,0))^{-1}\right)\right]^{-1} & \text{ otherwise, } \end{array}\right.$$
	and
	$$\delta_F(P,\lambda_0,r+1)= \left\{\begin{array}{ll} \frac{\left[\mu_{S_1,\norm{.}_{F}}^{r+1}\left((I_{r+1}\otimes M(\lambda_0;r))E~(T(P,\lambda_0))^{-1}\right)\right]^{-1}}{\sqrt{r+1}} & \text{ if } \lambda_0 \neq 0, \\
	\frac{\left[\mu_{S_2,\norm{.}_{F}}^{r+1}\left(E~(T(P,0))^{-1}\right)\right]^{-1}}{\sqrt{r+1}} & \text{ otherwise. }
	\end{array}\right.$$
\end{theorem}

\begin{proof} Recall that $S(P,\lambda_0)$ is the collection of all $n \times n$ matrix polynomials $\Delta P(\lambda)=\sum_{i=0}^{k} \lambda^i \Delta A_i$ such that $T(P + \Delta P, \lambda_0)$ has rank at most $(r+1)(n-1).$ As $P(\lambda_0)$ is invertible, $-\Delta P(\lambda) \in S(P,\lambda_0)$ if and only if
	\begin{equation}
	\mathrm{nullity}(I-T(\Delta P,\lambda_0)T(P,\lambda_0)^{-1})\geqslant r+1. \label{main_rel_mu}
	\end{equation}
	Due to \cref{fact_T_1} the above relation may be written as,
	\begin{equation}
	\mathrm{nullity}\left(
	I-\left(I_{r+1}\otimes\begin{bmatrix}\Delta P(\lambda_0)&\Delta P'(\lambda_0)&\cdots&\frac{\Delta P^{p}(\lambda_0)}{p!}\end{bmatrix}
	\right)
	E~(T(P,\lambda_0))^{-1}\right)\geqslant r+1.\label{rel_1_req_for_0}
	\end{equation}
	where $p = \min\{r,k\}.$
	By \cref{z_to_nz} this is equivalent to
	$$\mathrm{nullity}\left(
	I-\left(I_{r+1}\otimes\begin{bmatrix}\Delta A_0&\Delta A_1&\cdots&\Delta A_k\end{bmatrix}M(\lambda_0;r)
	\right)
	E~(T(P,\lambda_0))^{-1}\right)\geqslant r+1,$$
	which may also be written as
	\begin{equation}
	\mathrm{nullity}\left(
	I-\left(I_{r+1}\otimes[\Delta A_0 \cdots \Delta A_k]\right)\left( I_{r+1}\otimes M(\lambda_0;r)
	\right)
	E~(T(P,\lambda_0))^{-1}\right)\geqslant r+1 \label{dist_jchain_mu_1}.
	\end{equation}
	Then
	$$\inf\{\normp{\Delta P}_2| \Delta P(\lambda)\in S(P,\lambda_0)\}=\left[\mu_{S_1,\norm{.}_{2}}^{r+1}\left((I_{r+1}\otimes M(\lambda_0;r))E~(T(P,\lambda_0))^{-1}\right)\right]^{-1}.$$
	Similarly,
	\begin{equation*}\inf\{\normp{\Delta P}_F| \Delta P(\lambda)\in S(P,\lambda_0)\}=\frac{\left[\mu_{S_1,\norm{.}_{F}}^{r+1}\left((I_{r+1}\otimes M(\lambda_0;r))E~(T(P,\lambda_0))^{-1}\right)\right]^{-1}}{\sqrt{r+1}}.\end{equation*}
	The proof for the case $\lambda_0 \neq 0$ now follows from \cref{cor_sigma_result}.
	
	If $\lambda_0=0$, then the equation \cref{rel_1_req_for_0} will be of the form
	\begin{equation}\label{dist_jchain_mu_2}\mathrm{nullity}\left(
	I-\left(I_{r+1}\otimes\begin{bmatrix}\Delta A_0 & \cdots &\Delta A_p\end{bmatrix}
	\right)
	E~(T(P,0))^{-1}\right)\geqslant r+1.\end{equation}
	Then setting $\Delta A_i = 0$ for $i = r+1, \ldots, k$ if $r < k,$
	\begin{eqnarray*} \inf\{\normp{\Delta P}_2| \Delta P(\lambda)\in S(P,0)\} & = & \left[\mu_{S_2,\norm{.}_{2}}^{r+1}\left(E~(T(P,0))^{-1}\right)\right]^{-1} \text{ and }\\
		\inf\{\normp{\Delta P}_F| \Delta P(\lambda)\in S(P,0)\} & = & \frac{\left[\mu_{S_2,\norm{.}_{F}}^{r+1}\left(E~(T(P,0))^{-1}\right)\right]^{-1}}{\sqrt{r+1}}.\end{eqnarray*}
	Hence the proof follows from \cref{cor_sigma_result}. \end{proof}

\section{An alternative formulation of the distance as an optimization}

%
An alternative formulation for the distance $\delta_s(P,\lambda_0,r+1)$ is obtained in this section for $s =2$ or $F.$

\begin{theorem}\label{form_delta_F}
	Let $P(\lambda) = \sum_{i = 0}^k \lambda^iA_i$ be an $n \times n$ matrix polynomial of degree $k.$ For a given integer $r,$ such that $0 < r < kn,$ consider the sets
	\begin{eqnarray}
	\Gamma & := & \{[\gamma_1 \cdots \gamma_r]: \gamma_i > 0, 1 \leqslant i \leqslant r\}, \mbox{ and } \label{gamma}\\
	\mathbb{C}_0^{(r+1)n} & := &\{[x_0^T \cdots x_r^T]^T: x_i \in \mathbb{C}^n, i = 0,  \ldots, r, x_0 \neq 0\}. \label{c0rn}
	\end{eqnarray}
	Now let $\C_{\mathcal T, \Gamma}^{r,n}$ be the collection of all block Toeplitz like matrices $X$ given by,
	$$X= \left\{\begin{array}{ll}\scalemath{0.85}{\begin{bmatrix}x_0&x_1&x_2&\cdots&x_r\\0&\gamma_1x_0&\gamma_2x_1&\cdots&\gamma_rx_{r-1}\\0&0&\gamma_1\gamma_2{x_0}&\cdots&\gamma_{r-1}\gamma_r x_{r-2}\\\vdots&\vdots&\vdots&\vdots&\vdots\\0 &0&0&0&\gamma_1\gamma_2\dots\gamma_r{x_0}\end{bmatrix}} & \text{ if } r \leqslant k, \\
	\scalemath{0.85}{\begin{bmatrix}x_0&x_1&x_2&\cdots&x_k&\cdots&x_r\\0&\gamma_1x_0&\gamma_2x_1&\cdots&\gamma_kx_{k-1}&\cdots&\gamma_rx_{r-1}\\0&0&\gamma_1\gamma_2{x_0}&\cdots&\gamma_{k-1}\gamma_k{x_{k-2}}&\cdots&\gamma_{r-1}\gamma_r{x_{r-2}}\\\vdots&\vdots&\vdots&\vdots&\vdots&\vdots&\vdots\\0 &0&0&0&\displaystyle{\prod_{i = 1}^k\gamma_i} {x_0}&\cdots& \displaystyle{\prod_{i = r-k+1}^r\gamma_i}{x_{r-k}}\end{bmatrix}} & \text{ otherwise,}
	\end{array} \right.$$
	where $[x_0^T \cdots x_r^T]^T \in \mathbb{C}_0^{(r+1)n}$ and $[\gamma_1 \cdots \gamma_r] \in \Gamma.$ Then for $s=2$ or $F$
	\begin{equation}\label{label1}
	\delta_s(P,0,r+1)  = \left\{\begin{array}{ll} \displaystyle\inf_{X \in \C_{\mathcal T,\Gamma}^{r,n}}\norm{\begin{bmatrix}A_0&\cdots&A_r\end{bmatrix}XX^\dagger}_s & \text{ if } r \leqslant k, \\
	\displaystyle\inf_{X \in \C_{\mathcal T,\Gamma}^{r,n}}\norm{\begin{bmatrix}A_0&\cdots&A_k\end{bmatrix}XX^\dagger}_s & \text{ otherwise.}\end{array}\right.
	\end{equation}
	and in general,
	\begin{equation}\label{label2}
	\scalemath{0.9}{\delta_s(P,\lambda_0,r+1)  = \left\{ \begin{array}{ll} \displaystyle\inf_{X \in \C_{\mathcal T,\Gamma}^{r,n}}\norm{\begin{bmatrix}P(\lambda_0)&\cdots&\frac{1}{r!}P^r(\lambda_0)\end{bmatrix}X(M(\lambda_0;r)X)^\dagger}_s & \text{ if } r \leqslant k, \\
		\displaystyle\inf_{X \in \C_{\mathcal T,\Gamma}^{r,n}}\norm{\begin{bmatrix}P(\lambda_0)&\cdots&\frac{1}{k!}P^k(\lambda_0)\end{bmatrix}X(M(\lambda_0;r)X)^\dagger}_s & \text{ otherwise.} \end{array}\right.}
	\end{equation} where $M(\lambda_0;r)$ is as given in \cref{z_to_nz}.
\end{theorem}

\begin{proof} Initially consider the case that $r \leqslant k.$
	Let $\Delta P(\lambda) = \sum_{i=0}^k \lambda^i \Delta A_i$ be any $n \times n$ matrix polynomial of degree $k$ such that $P(\lambda)+\Delta P(\lambda)$
	is a regular matrix polynomial. Then $P(\lambda)+\Delta P(\lambda)$ has an elementary divisor $(\lambda-\lambda_0)^j$ where $j\geqslant r+1$ if and only if there exists vectors $x_0,x_1,\dots,x_r\in\mathbb{C}^n$ with $x_0\neq 0$ and $r$ positive scalars $\gamma_1,\dots,\gamma_r$ such that
	\begin{align*}
	&(P+\Delta P)(\lambda_0)x_0=0\\
	&(P+\Delta P)(\lambda_0)x_1+\gamma_1(P+\Delta P)^\prime(\lambda_0))x_0=0\\
	&(P+\Delta P)(\lambda_0)x_2+\gamma_2(P+\Delta P)^\prime(\lambda_0)x_1+\gamma_1\gamma_2\frac{(P+\Delta P)^{\prime\prime}(\lambda_0)}{2!}x_0=0\\
	&\dots\dots\dots\\
	&(P+\Delta P)(\lambda_0)x_r+\gamma_r(P+\Delta P)^\prime(\lambda_0)x_{r-1}+\dots+\displaystyle{\prod_{i = 1}^r}\gamma_i\frac{(P+\Delta P)^{r}(\lambda_0)}{r!}x_0=0.
	\end{align*}
	This is equivalent to
	$$\scalemath{0.82}{\begin{bmatrix}
		(P + \Delta P)(\lambda_0)&&&&\\
		\gamma_1 (P + \Delta P)^\prime(\lambda_0)& (P + \Delta P)(\lambda_0)&&&\\
		\frac{\gamma_1 \gamma_2}{2!}(P + \Delta P)^{\prime\prime}(\lambda_0)&\gamma_2 (P + \Delta P)^\prime(\lambda_0)& (P + \Delta P)(\lambda_0)&&\\
		\vdots&\vdots&&\ddots&\\
		\frac{\gamma_1\cdots\gamma_r}{r!} (P + \Delta P)^r(\lambda_0) & \frac{\gamma_2\cdots\gamma_r}{(r-1)!}
		(P + \Delta P)^{r-1}(\lambda_0)&\cdots&\cdots& (P + \Delta P)(\lambda_0)
		\end{bmatrix}
		\begin{bmatrix}x_0\\x_{1}\\\vdots\\x_{r-1}\\x_r\end{bmatrix}=0},$$
	which can be written in the form
	\begin{equation}\label{maineq}
	\scalemath{0.9}{ \begin{bmatrix}\Delta P(\lambda_0) & \Delta P^\prime(\lambda_0) &\cdots&\frac{1}{r!}\Delta P^r(\lambda_0)\end{bmatrix}
		X=-\begin{bmatrix} P(\lambda_0) & P^\prime(\lambda_0) &\cdots&\frac{1}{r!}P^r(\lambda_0)\end{bmatrix}X},
	\end{equation}
	where $$X=\begin{bmatrix}
	x_0&x_1&x_2&\cdots&x_r\\
	0&\gamma_1x_0&\gamma_2x_1&\cdots&\gamma_rx_{r-1}\\
	0&0&\gamma_1\gamma_2{x_0}&\cdots&\gamma_{r-1}\gamma_rx_{r-1}\\
	\vdots&\vdots&\vdots&\vdots&\vdots\\
	0&0&0&\cdots&\gamma_1\dots\gamma_r x_0
	\end{bmatrix} \in \C_{\mathcal{T},\gamma}^{r,n}.$$
	When $\lambda_0 = 0,$ \eqref{maineq} takes the form
	\begin{equation}\label{maineq_at_zero}
	\begin{bmatrix}\Delta A_0 & \Delta A_1 &\cdots&\Delta A_r\end{bmatrix}
	X=-\begin{bmatrix} A_0 & A_1 &\cdots&A_r\end{bmatrix}X.
	\end{equation}
	By~\cite[Lemma 1.3]{Sun93}, the minimum $2$ or Frobenius norm solution of this equation is given by
	\begin{equation}\label{min_sol_zero} \begin{bmatrix}\Delta A_0&\Delta A_1 &\cdots&\Delta A_r\end{bmatrix}=-\begin{bmatrix}A_0&A_1&\cdots&A_r\end{bmatrix}XX^\dagger.\end{equation}
	Setting $\Delta A_i = 0$ for $i = r+1, \ldots, k$ if $r < k,$ $(P+\Delta P)(\lambda)$ has an elementary divisor $\lambda^j, j \geq r+1.$ Therefore in this case, $\delta_s(P,0,r+1)=\displaystyle\inf_{X \in \C_{\mathcal{T},\Gamma}^{r,n}}\norm{\begin{bmatrix}A_0&\cdots&A_r\end{bmatrix}XX^\dagger}_s,$
	for $s=2$ or $F$. When $\lambda_0 \neq 0,$ \cref{z_to_nz} implies that
	$$\begin{bmatrix}\Delta P(\lambda_0) & \Delta P^\prime(\lambda_0) &\cdots&\frac{1}{r!}\Delta P^r(\lambda_0)\end{bmatrix} =
	\begin{bmatrix}\Delta A_0&\cdots&\Delta A_k\end{bmatrix}M(\lambda_0;r).$$
	Using this in \eqref{maineq}, the minimum $2$ or Frobenius norm solution of the resulting equation is given by
	$$\begin{bmatrix}\Delta A_0&\cdots&\Delta A_k\end{bmatrix}=-\begin{bmatrix} P(\lambda_0) & P^\prime(\lambda_0) &\cdots&\frac{1}{r!}P^r(\lambda_0)\end{bmatrix}X(M(\lambda_0;r)X)^\dagger,$$ thus proving that if $r \leqslant k$ then for $s=2$ or $F,$ $$\delta_s(P,\lambda_0,r+1)=\displaystyle\inf_{X \in \C_{\mathcal{T},\Gamma}^{r,n}}\norm{\begin{bmatrix} P(\lambda_0) & P^\prime(\lambda_0) &\cdots&\frac{1}{r!}P^r(\lambda_0)\end{bmatrix}X(M(\lambda_0;r)X)^\dagger}_s.$$
	
	If $r > k,$ then $P(\lambda)+\Delta P(\lambda)$ has an elementary divisor $(\lambda - \lambda_0)^j$ where $j\geqslant r+1$ iff there exists vectors $x_0,x_1,\dots,x_r\in\mathbb{R}^n$ with $x_0\neq 0$ and $r$ positive scalars $\gamma_1,\dots,\gamma_r$ such that
	\begin{align*}
	&\scalemath{0.85}{(P + \Delta P)(\lambda_0)x_0=0}\\
	&\scalemath{0.85}{(P +\Delta P)(\lambda_0)x_1+\gamma_1(P+\Delta P)^\prime(\lambda_0)x_0=0}\\
	&\scalemath{0.85}{(P+\Delta P)(\lambda_0)x_2+\gamma_2(P+\Delta P)^\prime(\lambda_0)x_1+\gamma_1\gamma_2\frac{(P+\Delta P)^{\prime\prime}(\lambda_0)}{2!}x_0=0}\\
	&\dots\dots\dots\\
	&\scalemath{0.85}{(P+\Delta P)(\lambda_0)x_k+\gamma_k(P+\Delta P)^\prime(\lambda_0)x_{k-1}+\dots+\displaystyle{\prod_{i = 1}^k}\gamma_i\frac{(P+\Delta P)^{k}(\lambda_0)}{k!}x_0=0}\\
	&\dots\dots\dots\\
	&\scalemath{0.85}{(P+\Delta P)(0)x_r+\gamma_r(P+\Delta P)^\prime(\lambda_0)x_{r-1}+\dots+\displaystyle{\prod_{i = r-k+1}^r}\gamma_i\frac{(P+\Delta P)^{k}(\lambda_0)}{k!}x_{r-k}=0.}
	\end{align*}
	This set of equations can be written in the form
	$$\scalemath{0.8}{\begin{bmatrix}
		(P + \Delta P)(\lambda_0)&&&&&\\
		\gamma_1 (P + \Delta P)^\prime(\lambda_0) & (P + \Delta P)(\lambda_0) &&&&\\
		\vdots&\vdots&\ddots&&\\
		\frac{\gamma_1 \cdots \gamma_k}{k!}(P + \Delta P)^k(\lambda_0)&&& (P + \Delta P)(\lambda_0)&&\\
		&\ddots&&&\ddots&\\
		&&&\frac{\gamma_{r-k+1} \cdots \gamma_r}{k!}(P + \Delta P)^k(\lambda_0)&\cdots & (P + \Delta P)(\lambda_0)
		\end{bmatrix}}
	\scalemath{0.8}{\begin{bmatrix}x_0\\x_{1}\\\vdots\\x_k\\\vdots\\x_r\end{bmatrix}=0.}$$
	This is equivalent to
	$$\scalemath{0.9}{\begin{bmatrix}\Delta P(\lambda_0) & \Delta P^\prime(\lambda_0) &\cdots&\frac{1}{k!}\Delta P^k(\lambda_0)\end{bmatrix}
		X=-\begin{bmatrix} P(\lambda_0) & P^\prime(\lambda_0) &\cdots&\frac{1}{k!}P^k(\lambda_0)\end{bmatrix}X},$$
	where $$X=\begin{bmatrix}x_0&x_1&x_2&\cdots&x_k&\dots&x_r\\0&\gamma_1x_0&\gamma_2x_1&\cdots&\gamma_kx_{k-1}&\cdots&\gamma_rx_{r-1}\\0&0&\gamma_1\gamma_2{x_0}&\cdots&\gamma_{k-1}\gamma_k{x_{k-2}}&\cdots&\gamma_{r-1}\gamma_r{x_{r-2}}\\\vdots&\vdots&\vdots&\vdots&\vdots&\vdots&\vdots\\0 &0&0&0&\displaystyle{\prod_{i=1}^k\gamma_i}{x_0}&\cdots&\displaystyle{\prod_{i=r-k+1}^r\gamma_i}{x_0}\end{bmatrix} \in \C_{\mathcal{T},\Gamma}^{r,n}.$$
	Therefore the proof follows by arguing as in the previous case.
\end{proof}

\begin{remark}
	The parameters $\gamma_i$ can all be taken to be $1$ in the optimization that computes $\delta_s(P,\lambda_0,r+1)$ for $s=2$ or $F.$ As shall be seen in~\cref{sec_upper}, this will also decrease the number of variables in the optimization. But these parameters play an important role when computing the upper bound for $\delta_s(P,\lambda_0,r+1)$ from this characterization. However, there is no particular advantage in choosing them to be nonzero real or complex numbers when deriving the upper bound.
\end{remark}

\section{Lower bounds on the distance}



The first lower bound on the distance $\delta_F(P,\lambda_0,r+1)$ is derived from \cref{sigma_result}.

\begin{theorem}\label{lb1} Let $P(\lambda) = \sum_{i=0}^k \lambda^iA_i$ be an $n \times n$ matrix polynomial of degree $k$ and $r < kn$ be a positive integer. For $\Gamma$ as given in \eqref{gamma}, let $\gamma := \left[\begin{array}{ccc} \gamma_1, \dots, \gamma_r \end{array}\right] \in \Gamma,$ and $T_\gamma(P,\lambda_0)$ be defined by \eqref{tgamma}. For $s=2$ or $F$, the distance $\delta_s(P,\lambda_0,r+1)$ to a nearest matrix polynomial having $(\lambda - \lambda_0)^j$ as an elementary divisor with $j \geqslant r+1$ satisfies
	$$\delta_s(P,\lambda_0,r+1) \geqslant \left\{ \begin{array}{ll} \sup_{\gamma \in \Gamma}\frac{f(\gamma)}{\sqrt{F_1(\gamma)}} & \text{ if } r \leqslant k,\\
	\sup_{\gamma \in \Gamma}\frac{f(\gamma)}{\sqrt {F_2(\gamma)}}  & \text{ otherwise,}
	\end{array}\right.$$
	\begin{align*}
	&\text{ where, }  f(\gamma) := \sigma_{(r+1)n-r}(T_{\gamma}(P,\lambda_0)),\\
	& F_1(\gamma) :=  \|M(\lambda_0;r)\|_2^2\max_{1\leqslant p \leqslant r}\left(1 + \sum_{t = 1}^p\prod_{i = t}^p \gamma_{(r+1-i)}^2\right) \text{ and }\\
	& \scalemath{0.85}{F_2(\gamma)  :=  \|M(\lambda_0;r)\|_2^2\max\left\{\max_{1 \leqslant p \leqslant k} \left(1 + \sum_{t=1}^p\prod_{i=t}^p\gamma_{(r+1-i)}^2\right), \max_{k+1 \leqslant p \leqslant r}\left(1 + \sum_{t = p-k+1}^p\prod_{i=t}^p\gamma_{(r+1-i)}^2\right)\right\}},
	\end{align*}
	$M(\lambda_0;r)$ being as defined in \cref{z_to_nz}.
\end{theorem}


\begin{proof}
	Consider the case that $r\leqslant k$. Let $\Delta P(\lambda) = \sum_{i=0}^k\lambda^i\Delta A_i$ be an $n \times n$ matrix polynomial such that $(P + \Delta P)(\lambda)$ has an elementary
	divisor $(\lambda-\lambda_0)^j$ with $j \geqslant r+1.$ By \cref{sigma_result},  $\sigma_{(r+1)n-r}(T_{\gamma}(P + \Delta P,\lambda_0))=0$ where $T_{\gamma}(P + \Delta P,\lambda_0)$
	is as defined in \eqref{tgamma}.
	By the perturbation theory for singular values,
	\begin{eqnarray*}
		f(\gamma) & = & \left|\sigma_{(r+1)n-r}(T_\gamma(P+\Delta P,\lambda_0)) - \sigma_{(r+1)n-r}(T_\gamma(P,\lambda_0)) \right| \\
		& \leqslant & \norm{T_\gamma(\Delta P,\lambda_0)}_2.
	\end{eqnarray*}
	Observe that
	$$T_\gamma(\Delta P,\lambda_0) = \left(I_{r+1}\otimes \begin{bmatrix}\Delta P(\lambda_0)&\cdots&\frac{1}{r!}\Delta P^r(\lambda_0)\end{bmatrix}\right)\begin{bmatrix}\hat E_{\gamma,1}^T&\hat E_{\gamma,2}^T\cdots&\hat E_{\gamma,r+1}^T\end{bmatrix}^T$$
	where $\hat E{\gamma,_j} = E_{\gamma,j} \otimes I_n,$  $E_{\gamma,j}, j = 1, \ldots, r+1,$ being the $(r+1) \times (r+1)$ matrices \\
	$E_{\gamma,1} =\begin{bmatrix}1&&&&\\&&&&\\&&&&\\&&&&\\&&&&\end{bmatrix}, E_{\gamma,2} =\begin{bmatrix}&1&&&\\\gamma_1&&&&\\&&&&\\ &&&&\\&&&&\end{bmatrix}, \dots,
	E_{\gamma,r}  =  \scalemath{0.85}{\begin{bmatrix}&&&1&\\&&\gamma_{r-1}&&\\&\iddots&&&\\ \displaystyle{\prod_{i=1}^{r-1}\gamma_i}&&&&\\&&& &\end{bmatrix}},\\
	E_{\gamma,r+1} =  \scalemath{0.85}{\begin{bmatrix}&&&&1\\&&&\gamma_r&\\&&\gamma_r \gamma_{r-1}&&\\&\iddots&&&\\\displaystyle{\prod_{i=1}^r\gamma_i}&&&&\end{bmatrix}}.$ \\
	$$\text{Therefore, } f(\gamma) \leqslant \norm{\begin{bmatrix}\Delta P(\lambda_0)&\cdots&\frac{1}{r!}\Delta P^r(\lambda_0)\end{bmatrix}}_2\norm{\begin{bmatrix}\hat E_{\gamma,1}^T&\hat E_{\gamma,2}^T&\cdots&\hat E_{\gamma,r+1}^T\end{bmatrix}^T}_2.$$
	Now, $$\scalemath{1}{ \norm{\begin{bmatrix}\Delta P(\lambda_0) & \cdots & \frac{1}{r!}\Delta P^r(\lambda_0)\end{bmatrix}}_2 \leqslant \normp{\Delta P}_2\|M(\lambda_0;r)\|_2,}$$ where the last inequality holds due to \cref{z_to_nz}.
	This shows for $s=2$ or $F$, $$\frac{f(\gamma)}{\sqrt{F_1(\gamma)}}\leqslant \delta_s(P,\lambda_0,r+1),$$
	and the proof follows by taking the supremum of the left hand side as $\gamma$ varies over $\Gamma$. The lower bound for the case $r>k$ can be proved by similar arguments.
\end{proof}

A second lower bound is obtained from the formulation of $\delta_F(P,\lambda_0,r+1)$ as the reciprocal of a generalized $\mu$-value as derived in \cref{dist_as_mv}. The following lemma, the proof of which is evident from elementary properties of singular values, will be used to derive this bound.

\begin{lemma}\label{lem_g_mu_value}
	Let $M\in \mathbb{C}^{p\times q}$ and $N\in \mathbb{C}^{q\times p}$. Then for $i=1,\dots,\min\{p,q\}$
	$$\mathrm{inf}\{\|M\|_2: \mathrm{rank}(I_p-MN)\leqslant p-i\}=(\sigma_i(N))^{-1}.$$
\end{lemma}

\begin{theorem}\label{lb2}
	Let $P(\lambda)=\sum_{i=0}^{k}\lambda^i A_i$ be an $n \times n$ matrix polynomial of degree $k,$ and $1 \leqslant r < kn.$
	For any $\lambda_0 \in \C$ which is not an eigenvalue of $P(\lambda),$ let $T(P,\lambda_0)$ be as defined by~\eqref{blocktoep} and $E$ and $M(\lambda_0;r)$ be as defined in
	\cref{fact_T_1} and \cref{z_to_nz} respectively. For a given positive integer $t,$ let $\mathbb{S}^t$ be the collection of column vectors of length $t$ with positive entries and $$W_{t,n}(a):=\begin{bmatrix}a_1&&\\&\ddots&\\&&a_t\end{bmatrix}\otimes I_n,$$
	where $a=\begin{bmatrix}a_1,\cdots,a_t\end{bmatrix}^T \in \mathbb{S}^t.$
	Then for $s = 2$ or $F,$ $$\delta_s(P,\lambda_0,r+1)\geqslant\sup_{a\in\mathbb{S}^{r+1}} (\sigma_{r+1}(B(\lambda_0,a,P)))^{-1},$$
	where $B(\lambda_0,a,P) :=  W_{r+1,(k+1)n}(a)(I_{r+1}\otimes M(\lambda_0;r))E~(T(P,\lambda_0))^{-1}W_{r+1,n}^{-1}(a).$
	
	If $\lambda_0=0$ then for $s = 2,$ or $F,$ $$\delta_s(P,0,r+1)\geqslant\sup_{a\in\mathbb{S}^{r+1}} \sigma_{r+1}((\widehat{B}(0,a,P)))^{-1}.$$
	for $\widehat{B}(0,a,P) := W_{r+1,({\min\{r,k\}}+1)n}(a)E~(T(P,0))^{-1}W_{r+1,n}^{-1}(a).$
	
\end{theorem}

\begin{proof}
	If $-\Delta P(\lambda) \in S(P, \lambda_0),$ from inequality~\cref{dist_jchain_mu_1} the nullity of
	$$I-W_{r+1,n}^{-1}(a)\left(I_{r+1}\otimes [\Delta A_0 \cdots \Delta A_k]
	\right)
	W_{r+1,(k+1)n}(a)(I_{r+1}\otimes M(\lambda_0;r)) E~(T(P,\lambda_0))^{-1}$$
	is at least $r+1.$
	\noindent Therefore
	$\mathrm{nullity} \left(I-\left(I_{r+1}\otimes [\Delta A_0 \cdots \Delta A_k]\right)B(\lambda_0,a,P)\right)\geqslant r+1.$
	By ~\cref{lem_g_mu_value},
	$\norm{\begin{bmatrix}\Delta A_0&\cdots&\Delta A_k\end{bmatrix}}_2\geqslant (\sigma_{r+1}(B(\lambda_0,a,P)))^{-1},$
	so that for $s = 2$ or $F,$
	$$\delta_s(P,\lambda_0,r+1)\geqslant\sup_{a\in\mathbb{S}^{r+1}} (\sigma_{r+1}(B(\lambda_0,a,P)))^{-1}.$$
	
	Similarly if $-\Delta P(\lambda_0) \in S(P,0),$ from inequality~\cref{dist_jchain_mu_2} the nullity of
	$$I-W_{r+1,n}^{-1}(a)\left(I_{r+1}\otimes \begin{bmatrix}\Delta A_0&\cdots&\Delta A_p\end{bmatrix}
	\right)W_{r+1,(p+1)n}(a) E~(T(P,0))^{-1}$$
	is at least $r+1,$ where $p = \min\{r,k\}.$ This is equivalent to
	$$\mathrm{nullity}\left(I-\left(I_{r+1}\otimes \begin{bmatrix}\Delta A_0&\cdots&\Delta A_p\end{bmatrix}\right) \widehat{B}(0,a,P)\right) \geqslant r+1.$$
	\noindent By ~\cref{lem_g_mu_value},
	$\norm{\begin{bmatrix}\Delta A_0&\cdots&\Delta A_p\end{bmatrix}}_2\geqslant (\sigma_{r+1} (\widehat{B}(0,a,P)))^{-1}.$
	\noindent Therefore
	$$\delta_s(P,0,r+1)\geqslant\sup_{a\in\mathbb{S}^{r+1}} (\sigma_{r+1}(\widehat{B}(0,a,P)))^{-1} \text{ for } s = 2 \text{ or } F.$$
\end{proof}

\section{Upper bound on the distance}\label{sec_upper}

In this section an upper bound on the distance $\delta_F(P,\lambda_0,r+1)$ that can be used in conjunction with the lower bound obtained in \cref{lb1} is derived.

\begin{theorem}\label{ubound} Let $P(\lambda) = \sum_{i=0}^k \lambda^iA_i$ be an $n \times n$ matrix polynomial of degree $k$ and $r < kn$ be a positive integer. For $\Gamma$ as given in \eqref{gamma}, let $\gamma := \left[\begin{array}{ccc} \gamma_1, \dots, \gamma_r \end{array}\right] \in \Gamma,$ and let $f(\gamma) := \sigma_{(r+1)n-r}(T_\gamma(P,\lambda_0))$ where $T_\gamma(P,\lambda_0)$ is as defined in \eqref{tgamma}. Suppose that $v(\gamma)=\begin{bmatrix}v_0^T&v_1^T&\cdots&v_r^T\end{bmatrix}^T$ and $u(\gamma)=\begin{bmatrix}u_0^T&u_{1}^T&\cdots&u_r^T\end{bmatrix}^T$ are the corresponding right and left singular vectors with $v_i, u_i \in \C^n$ for $i = 0,1, \ldots, r$ dependent on $\gamma.$ Also let $\Gamma_0 \subset \Gamma$ be the collection of all $\gamma \in \Gamma$ with the property that the vector $v_0$ formed by the first $n$ entries of a right singular vector $v(\gamma)$ associated with the singular value $f(\gamma)$  of $T_\gamma(P,\lambda_0)$ is nonzero. Then for $s=2$ or $F$
	\begin{equation} \delta_s(P,\lambda_0,r+1) \leqslant \inf_{\gamma \in \Gamma_0} f(\gamma)\norm{ U(\gamma)(M(\lambda_0;r)V(\gamma))^\dagger}_s \end{equation}
	where $M(\lambda_0;r)$ is as defined in \cref{z_to_nz},  $U(\gamma) = \begin{bmatrix}u_0&\cdots &u_r\end{bmatrix},$
	$$V(\gamma) :=  \scalemath{0.9}{\left\{\begin{array}{ll} \begin{bmatrix}
		v_0&v_1&\cdots&v_r\\
		0&\gamma_1v_0&\cdots&\gamma_rv_{r-1}\\
		0&0&\cdots&\gamma_{r-1}\gamma_rv_{r-2}\\
		\vdots&\vdots&\vdots&\vdots\\
		0&0&\cdots&\displaystyle{\prod_{i=1}^r\gamma_i} v_0
		\end{bmatrix}^\dagger &  \text{ if } r \leqslant k, \\
		\begin{bmatrix}
		v_0&v_1&v_2&\cdots&v_k&\cdots&v_r\\
		0&\gamma_1v_0&\gamma_2v_1&\cdots&\gamma_kv_{k-1}&\cdots&\gamma_rv_{r-1}\\
		0&0&\gamma_1\gamma_2{v_0}&\cdots&\gamma_{k-1}\gamma_kv_{k-2}&\cdots&\gamma_{r-1}\gamma_rv_{r-2}\\
		\vdots&\vdots&\vdots&\vdots&\vdots&\vdots&\vdots\\
		0 &0&0&0&\displaystyle{\prod_{i=1}^k\gamma_i}v_0&\cdots & \displaystyle{\prod_{i=r-k+1}^r\gamma_i}v_{r-k}\end{bmatrix}^\dagger &  \text{ otherwise, }
		\end{array}\right.}$$
	and the infimum is taken to be $\infty$ if $\Gamma_0 = \emptyset.$
\end{theorem}

\begin{proof} From equations~\eqref{label1} and \eqref{label2} it is clear that if $\gamma \in \Gamma_0,$ then $V(\gamma)$ satisfies
	
	$$\scalemath{0.9}{\delta_s(P,\lambda_0,r+1)  \leqslant \left\{\begin{array}{ll} \norm{\begin{bmatrix}P(\lambda_0)& \cdots & \frac{1}{r!}P^r(\lambda_0)\end{bmatrix}V(\gamma)(M(\lambda_0;r)V(\gamma))^\dagger}_s & \text{ if } r \leqslant k, \\[0.5ex]
		\norm{\begin{bmatrix}P(\lambda_0)& \cdots & \frac{1}{k!}P^k(\lambda_0)\end{bmatrix}V(\gamma)(M(\lambda_0;r)V(\gamma))^\dagger}_s & \text {otherwise }\end{array}\right.} $$
	for $s=2$ or $F$. As $v(\gamma)$ is a right singular vector of $T_\gamma(P,\lambda_0)$ corresponding to $f(\gamma),$ it is clear that $\begin{bmatrix}P(\lambda_0)& \cdots & \frac{1}{p!}P^p(\lambda_0)\end{bmatrix}V(\gamma) = U(\gamma)$ where $p = \min\{r,k\}$.  In either case,
	$\delta_s(P,\lambda_0,r+1) \leqslant f(\gamma)\norm{ U(\gamma)(M(\lambda_0;r)V(\gamma))^\dagger}_s,$
	and the proof follows by taking the infimum of the right hand side of the above inequality as $\gamma$ varies over $\Gamma_0.$
\end{proof}

\begin{remark} A matrix polynomial for which $\Gamma_0 = \emptyset$ has never been encountered in practice. Therefore it is conjectured that the upper bound in ~\cref{ubound} is never $\infty.$ In fact numerical experiments show that in many cases this upper bound is very close to the computed value of the distance.
\end{remark}

\section{Some special cases}

The quantities $\delta_s(P,0,2), s = 2, F,$ are a measure of the distance to a matrix polynomial nearest to $P(\lambda) = \sum_{i=0}^k \lambda^i A_i,$  having a defective eigenvalue at $0.$ In this case the problem is equivalent to finding a nearest matrix pencil to $\lambda A_1 + A_0$ in the chosen norm that has $0$ as a defective eigenvalue. Note that this distance is of significant practical interest as when $P(\lambda)$ is replaced by $\mathrm{rev} \, P(\lambda),$ then it is the distance to a nearest matrix polynomial with a defective eigenvalue at $\infty.$ This problem was considered in \cite{KotDBB19} for the matrix pencils where several results that apply only to this special case were obtained.

Firstly, the upper bound for the distances in \cref{ubound} is given by $$\displaystyle\inf_{\gamma \in \Gamma_0}f(\gamma)\norm{\begin{bmatrix}u_0& u_1\end{bmatrix}\begin{bmatrix}v_0 & v_1\\0& \gamma v_0\end{bmatrix}^\dagger}_s$$ for $s=2$ or $F$ where $\begin{bmatrix}v_0\\v_1\end{bmatrix}$ and $\begin{bmatrix}u_0\\u_1\end{bmatrix}$ are the right and left singular vectors of $\begin{bmatrix}P(0) & 0\\\gamma P^\prime(0) & P(0) \end{bmatrix}$ corresponding to its $(2n-1)$th singular value $f(\gamma)$. In this case, $\gamma$ can be allowed to vary over all positive real numbers as the restriction $v_0 \neq 0$ can be removed. To see this, assume that $\gamma > 0$ is such that the corresponding vector $v_0 = 0.$  Then clearly $u_0 = 0$ and $f(\gamma)\norm{\begin{bmatrix}0& u_1\end{bmatrix}\begin{bmatrix}0& v_1\\0& 0\end{bmatrix}^\dagger}_s=f(\gamma).$ Let $\Delta P(\lambda) = \sum_{i=0}^k \lambda^i \Delta A_i$ where $\Delta A_0 = -f(\gamma)u_1v_1^*$ and $\Delta A_i = 0$ for all $i = 2, \ldots, k.$ Then $$\scalemath{0.9}{[\Delta A_0 \quad \Delta A_1] = -f(\gamma)\begin{bmatrix}0& u_1\end{bmatrix}\begin{bmatrix}0& v_1\\0& 0\end{bmatrix}^\dagger \text{ and } \normp{\Delta P}_s = f(\gamma)\left\|\begin{bmatrix}0& u_1\end{bmatrix}\begin{bmatrix}0& v_1\\0& 0\end{bmatrix}^\dagger\right\|_s = f(\gamma)}$$
for $s = 2,$ or $F$ and the relations
$$\begin{bmatrix} P(0) & \\ \gamma P^\prime(0) & P(0) \end{bmatrix} \begin{bmatrix} 0 \\ v_1 \end{bmatrix} = f(\gamma) \begin{bmatrix} 0 \\ u_1 \end{bmatrix} \text{ and } \begin{bmatrix} 0 & u_1^* \end{bmatrix}\begin{bmatrix} P(0) & \\ \gamma P^\prime(0) & P(0) \end{bmatrix} = f(\gamma) \begin{bmatrix} 0 & v_1^* \end{bmatrix},$$
imply that $A_0v_1 = f(\gamma)u_1, u_1^*A_0 = f(\gamma)v_1^*$ and $u_1^*A_1 = 0.$ Therefore,
$$u_1^*(P + \Delta P)(0) = u_1^*A_0 - f(\gamma)v_1^* = 0, \, (P + \Delta P)(0)v_1 = A_0v_1 - f(\gamma)u_1 = 0$$ and $u_1^*(P + \Delta P)'(0)v_1 = u_1^*A_1v_1 = 0.$ So unless $(P + \Delta P)(\lambda)$ is singular, $0$ is a multiple eigenvalue of $(P + \Delta P)(\lambda).$ In either case the objective is achieved as the polynomial $(P + \Delta P)(\lambda)$ is arbitrarily close to having an elementary divisor $\lambda^j, \, j \geq 2.$

Secondly, a formula for the Frobenius norm distance to a nearest matrix polynomial with a defective eigenvalue at $0$ may be found for the special case that $0$ is already and eigenvalue of $P(\lambda)$ (so that $\mathrm{rank} \, A_0 = n-1$) and the allowable perturbations to $P(\lambda)$  have the property that their coefficient matrices have rank atmost $1.$ The formula is given by the following theorem, the proof of which is identical to that of~\cite[Theorem 5.4]{KotDBB19}.

\begin{theorem}
	Let $P(\lambda)=\sum_{i=0}^k \lambda^i A_i$ be an $n \times n$ matrix polynomial of degree $k$ where $\mathrm{rank} A_0 = n-1$. Suppose $A_0 = U \Sigma V^*$ is
	a Singular Value Decomposition (SVD) of $A_0$ and $a_{i,j}$ is the entry of $U^*A_1V$ in the $i^{\rm th}$ row and $j^{\rm th}$ column.
	Define $X$ and $Y$ as:
	\[
	{X\!\!  :=\!\!   \begin{bmatrix}\sigma_{1}   & \ldots & \ldots & \ldots &0\\
		0  &  \sigma_{2} & \ldots & \ldots &0\\
		\vdots & \vdots &  \ddots  & \vdots & \vdots \\
		\vdots   & \ldots & 0 &\sigma_{n-1}  & 0\\
		a_{n,1}  & \ldots &\ldots & a_{n,n-1} & a_{n.n}\\ \end{bmatrix} \!\! ,
		Y \!\! :=\!\!   \begin{bmatrix}\sigma_{1}   & \ldots & \ldots & \ldots & a_{1,n}\\
		0  &  \sigma_{2} & \ldots & \ldots  & a_{2,n}\\
		\vdots & \vdots  & \ddots  & \vdots & \vdots \\
		\vdots   & \ldots & 0 &\sigma_{n-1}  & a_{n-1,n}\\
		0  & \ldots &\ldots & 0 & a_{n.n}\\ \end{bmatrix} \!\! .}
	\]
	where $\sigma_1\geqslant \sigma_2\geqslant \cdots \geqslant \sigma_{n-1} > 0$ are the singular of $A_0.$
	Then, the distance with respect to the norm $\normp{ \cdot }_F,$ to the nearest matrix polynomial with a defective eigenvalue at $0$ under the restriction that the coefficient matrices of the perturbating matrix polynomial have rank at most $1$ is given by $\min\{\sigma_{\min}(X),\sigma_{\min}(Y)\}.$
\end{theorem}

\section{Numerical Experiments}

This section presents numerical experiments conducted to illustrate the upper and lower bounds on the distances and their values computed  via BFGS and MATLAB's {\tt globalsearch} algorithm from the formulation in \cref{form_delta_F}.
Computing $\delta_F(P,\lambda_0,r+1)$ from the optimization in \cref{form_delta_F} via BFGS requires the gradient of the objective function $f(X):=\norm{HX(M(\lambda_0;r)X)^\dagger}_F$ where $X$ varies depending on whether $r\leqslant k$ or $r>k$ and $H = \begin{bmatrix} P(\lambda_0) & \cdots & \frac{1}{p!}P^p(\lambda_0) \end{bmatrix},$ $p=\min\{r,k\}.$ By \cref{z_to_nz},
$H = \begin{bmatrix} A_0 & \cdots & A_k \end{bmatrix}M(\lambda_0;r).$
Therefore $$f(X) = \norm{G(M(\lambda_0;r)X)(M(\lambda_0;r)X)^\dagger}_F$$ where $G = \begin{bmatrix} A_0 & \cdots & A_k \end{bmatrix}.$ Only real matrix polynomials are considered in the experiments. Since $M(\lambda_0;r)$ has full column rank, for any $X = X_0$ if there exists a neighborhood $S$ of $X_0$ such that $\mathrm{rank} X_0=\mathrm{rank} X$ for all $X\in S$ then $f(X)$ is differentiable at $X_0$. If we use any numerical scheme to find the infimum of $f(X),$ then generically at every step there exists a neighborhood $S$ of $X$ where every element of $S$ is of full rank and consequently we can find gradient of $f(X)$ at those points. Additionally the matrix $X$ involved in the objective function $f(X)$ has block Toeplitz structure which needs to be incorporated when finding the gradient of $f(X)$. For simplicity, the gradient is initially computed for the function $(f(X))^2$ without taking the structure of $X$ into consideration with the changes due to the structure being incorporated later. Therefore the function under consideration is
$$g(X):=(f(X))^2=\norm{G(M(\lambda_0;r)X)(M(\lambda_0;r)X)^{\dagger}}_F^2.$$
Considering $g(X)$ as a real valued function of the entries of $X,$ $$\nabla g(X)\Big |_{X=X_0}=\mathrm{vec}\left(\frac{dg}{dX}\Big |_{X=X_0}\right).$$ Now, setting $Y = M(\lambda_0;r)X,$
$$dg = 2\inner{GYY^{\dagger}}{Gd(YY^{\dagger})} \text{ where } \inner{A}{B} = \mathrm{trace} A^TB.$$
Expanding the right hand side gives \begin{equation}\label{grad_cal1} dg = 2\inner{G^TGYY^{\dagger}(Y^{\dagger})^T}{dY}+2\inner{Y^TG^TGYY^{\dagger}}{dY^{\dagger}} \end{equation}
where,
$$dY^{\dagger}=(I-Y^{\dagger}Y)dY^T(Y^{\dagger})^T(Y^{\dagger})+(Y^{\dagger})(Y^{\dagger})^TdY^T(I-YY^{\dagger})-(Y^{\dagger})dY(Y^{\dagger}).$$
Therefore,
\begin{align*}
dg= &2\inner{G^TGYY^{\dagger}(Y^{\dagger})^T}{dY}+2\inner{Y^TG^TGYY^{\dagger}}{(I-Y^{\dagger}Y)dY^T(Y^{\dagger})^T(Y^{\dagger})}\\
&+2\inner{Y^TG^TGYY^{\dagger}}{(Y^{\dagger})(Y^{\dagger})^TdY^T(I-YY^{\dagger})}-2\inner{Y^TG^TGYY^{\dagger}}{(Y^{\dagger})dY(Y^{\dagger})}\\
=&2\inner{G^TGYY^{\dagger}(Y^{\dagger})^T}{dY}+2\inner{(I-Y^{\dagger}Y)^TY^TG^TGYY^{\dagger}(Y^{\dagger})^T(Y^{\dagger})}{dY^T}\\
&+2\inner{(Y^{\dagger})(Y^{\dagger})^TY^TG^TGYY^{\dagger}(I-YY^{\dagger})^T}{dY^T}\\
&-2\inner{(Y^{\dagger})^TY^TG^TGYY^{\dagger}(Y^{\dagger})^T}{dY}\\
=&2\inner{\mathcal{F}(G,Y)}{dY}
\end{align*}
where
\begin{align*}
\mathcal{F}(G,Y) := & G^TGYY^{\dagger}(Y^{\dagger})^T+(Y^{\dagger})^TY^{\dagger}(Y^{\dagger})^TY^TG^TGY(I-Y^{\dagger}Y)\\
& +(I-YY^{\dagger})(Y^{\dagger})^TY^TG^TGYY^{\dagger}(Y^{\dagger})^T-(Y^{\dagger})^TY^TG^TGYY^{\dagger}(Y^{\dagger})^T.
\end{align*}
As, $dY = M(\lambda_0;r)X,$ $\frac{dg}{dX} = 2M(\lambda_0;r)^T\mathcal{F}(G,Y)=:\psi(X).$
Now at a fixed $X_0$, $\frac{dg}{dX}\Big |_{X=X_0}=\psi(X_0).$ Due to the structure of $X,$ $\nabla g(X)\Big |_{X=X_0}$ is given by
$$\scalemath{0.85}{\nabla g(X)\Big |_{X=X_0}=
	\begin{bmatrix}\sum_{i=1}^{r+1}\left(\begin{bmatrix} \psi(X_0)_{(i-1)n+1,i} \\ \vdots \\ \psi(X_0)_{in,i} \end{bmatrix}\right) \\
	\sum_{i=1}^{r}\left(\begin{bmatrix} \psi(X_0)_{(i-1)n+1,i+1} \\ \vdots \\ \psi(X_0)_{in,i+1} \end{bmatrix}\right) \\ \vdots \\ \begin{bmatrix} \psi(X_0)_{1,r+1} \\ \vdots \\ \psi(X_0)_{n,r+1} \end{bmatrix} \end{bmatrix} \text{ if } r\leqslant k,}$$
and by
$$\scalemath{0.85}{\nabla g(X)\Big |_{X=X_0}=\begin{bmatrix}\sum_{i=1}^{k+1} \left(\begin{bmatrix} \psi(X_0)_{(i-1)n+1,i} \\ \vdots \\ \psi(X_0)_{in,i} \end{bmatrix}\right) \\ \vdots \\ \sum_{i=1}^{k+1} \left(\begin{bmatrix} \psi(X_0)_{(i-1)n+1,i+r-k} \\ \vdots \\ \psi(X_0)_{in,i+r-k} \end{bmatrix}\right) \\ \sum_{i=1}^{k} \left(\begin{bmatrix} \psi(X_0)_{(i-1)n+1,i+r-k+1} \\ \vdots \\ \psi(X_0)_{in,i+r-k+1} \end{bmatrix}\right) \\ \vdots\\ \begin{bmatrix} \psi(X_0)_{1,r+1} \\ \vdots \\ \psi(X_0)_{n,r+1} \end{bmatrix} \end{bmatrix} \text{ if }  r>k.}$$

Due to the difficulties in computing the gradient of the objective function, the optimization for $\delta_2(P,\lambda_0,r+1)$ in \cref{form_delta_F} is performed only via MATLAB's {\tt globalsearch.m}. Also in each case, the optimizations involved in the lower and upper bounds are computed via {\tt globalsearch.m} algorithm.

\begin{example}\label{ex1}
	\rm{Consider a $2 \times 2$ matrix polynomial of degree $3$, $$\scalemath{0.85}{P(\lambda)=\begin{bmatrix}-.1414&-.1490\\1.1928&.9702\end{bmatrix}+\lambda \begin{bmatrix}.8837&.9969\\.2190&.0259\end{bmatrix}+\lambda^2\begin{bmatrix}.6346&.9689\\.6252&-.0649\end{bmatrix}+\lambda^3\begin{bmatrix}-1.9867&1.2800
		\\.6097&-.1477\end{bmatrix}}.$$
	\cref{tb11} records the the values of the distance $\delta_F(P,0,r)$ computed via {\tt globalsearch.m} and BFGS algorithms using the formulation in \cref{form_delta_F} for various values of $r$ together with lower bounds from \cref{lb1} and \cref{lb2} and the upper bound from \cref{ubound}. \cref{tb12} records the same for  $\delta_F(P,1,r)$ as $r$ varies from $2$ to $6.$
	Likewise,  \cref{tb13} and \cref{tb14} records the corresponding quantities for the distances $\delta_2(P,0,r)$ and $\delta_2(P,1,r)$ respectively, except that in these cases the distance is computed only via the {\tt globalsearch.m} algorithm.}
	
	\begin{table}[h!]
		\begin{center}
			\begin{tabular}{|c|c|c|c|c|c|}
				\hline
				\rotanti{Distance} \rotanti{measured} & \rotanti{Lower bound}~  \rotanti{(\cref{lb2})}  & \rotanti{Lower bound} ~\rotanti{(\cref{lb1})}   & \rotanti{BFGS} & \rotanti{{\tt globalsearch}} & \rotanti{Upper bound}~ \rotanti{(\cref{ubound})}\\ \hline
				$\delta_F(P,0,2)$ &  0.10797922 &  0.10683102    & 0.14992951  & 0.14992951         & 0.1504944      \\ \hline
				$\delta_F(P,0,3)$ & 0.17943541 &  0.17354340    & 0.27433442  & 0.27433442         & 0.27996519     \\ \hline
				$\delta_F(P,0,4)$ & 0.83444419 &  0.65889251    & 1.41424988 &  1.41424988       & 1.4189444     \\ \hline
				$\delta_F(P,0,5)$ & 0.90827444 &  0.75348431    & 1.46326471 &  1.46326471       & 1.47185479     \\ \hline
				$\delta_F(P,0,6)$ & 0.99263034 &  0.85789363   & 1.66359899 &  1.66359899       & 1.72452708     \\
				\hline
			\end{tabular}
			\caption{\label{tb11}Comparison of upper and lower bounds with the distance $\delta_F(P,0,r)$ calculated by BFGS and {\tt globalsearch.m} for \cref{ex1}.}
		\end{center}
	\end{table}

	\begin{table}[h!]
		\begin{center}
			\begin{tabular}{|c|c|c|c|c|c|}
				\hline
				\rotanti{Distance} \rotanti{measured} & \rotanti{Lower bound}~  \rotanti{(\cref{lb2})}  & \rotanti{Lower bound}~ \rotanti{(\cref{lb1})}   & \rotanti{BFGS} & \rotanti{{\tt globalsearch}} & \rotanti{Upper bound}~ \rotanti{(\cref{ubound})}\\ \hline
				$\delta_F(P,1,2)$ & 1.35798224 & 0.70551994   &1.35814780 &1.35814780  & 1.39370758     \\ \hline
				$\delta_F(P,1,3)$ & 1.35690676 & 0.57675049   &1.42078740 &1.42078740  &1.57015806     \\ \hline
				$\delta_F(P,1,4)$ & 1.35798160 & 0.56881053   &1.42220397 &1.42220397  &1.76028594     \\ \hline
				$\delta_F(P,1,5)$ & 1.35689708 & 0.56908887   &1.45865399 &1.45865399  &1.82967789      \\ \hline
				$\delta_F(P,1,6)$ & 1.35690633 & 0.56789237   &1.46349849 &1.46349849  &1.57008146     \\
				\hline
			\end{tabular}
			\caption{\label{tb12}Comparison of upper and lower bounds with the distance $\delta_F(P,1,r)$ calculated by BFGS and {\tt globalsearch.m} for \cref{ex1}.}
		\end{center}
	\end{table}
	
	\begin{table}[h!]
		\begin{center}
			\begin{tabular}{|c|c|c|c|c|c|}
				\hline
				\rotanti{Distance} \rotanti{measured}& \rotanti{Lower bound} ~ \rotanti{(\cref{lb2})}  & \rotanti{Lower bound}  ~\rotanti{(\cref{lb1})}  & \rotanti{{\tt globalsearch}} & \rotanti{Upper bound}~ \rotanti{(\cref{ubound})}\\ \hline
				$\delta_2(P,0,2)$ &  0.10797922 &  0.10683102    &.10797922  & 0.11368413       \\ \hline
				$\delta_2(P,0,3)$ & 0.17943541 &  0.17354340    &.19516063 & 0.21687613    \\ \hline
				$\delta_2(P,0,4)$ & 0.83444419 &  0.65889251    &1.04436762 & 1.05968598      \\ \hline
				$\delta_2(P,0,5)$ & 0.90827444 &  0.75348431    &1.13265970 & 1.20943709        \\ \hline
				$\delta_2(P,0,6)$ & 0.99263034 &  0.85789363    &1.55726928 & 1.70199290       \\
				\hline
			\end{tabular}
			\caption{\label{tb13}Comparison of upper and lower bounds with the distance $\delta_2(P,0,r)$ calculated by {\tt globalsearch.m} for \cref{ex1}.}
		\end{center}
	\end{table}

	\begin{table}[h!]
		\begin{center}
			\begin{tabular}{|c|c|c|c|c|c|}
				\hline
				\rotanti{Distance} \rotanti{measured} & \rotanti{Lower bound} ~\rotanti{(\cref{lb2})}  & \rotanti{Lower bound} ~\rotanti{(\cref{lb1})}   & \rotanti{{\tt globalsearch}} & \rotanti{Upper bound} ~\rotanti{(\cref{ubound})}\\ \hline
				$\delta_2(P,1,2)$ & 1.35798224 & 0.70551994   &1.35798224 & 1.35827634     \\ \hline
				$\delta_2(P,1,3)$ & 1.35690676 & 0.57675049   &1.35805109 & 1.35813196    \\ \hline
				$\delta_2(P,1,4)$ & 1.35798160 & 0.56881053   &1.35805159 & 1.56108421      \\ \hline
				$\delta_2(P,1,5)$ & 1.35689708 & 0.56908887   &1.35805160 & 1.52575381    \\ \hline
				$\delta_2(P,1,6)$ & 1.35690633 & 0.56789237   &1.416503376 & 1.43921050      \\
				\hline
			\end{tabular}
			\caption{\label{tb14}Comparison of upper and lower bounds with the distance $\delta_2(P,1,r)$ calculated by {\tt globalsearch.m} for \cref{ex1}.}
		\end{center}
	\end{table}
\end{example}

\begin{example}\label{ex2}
\rm{	Consider the matrix polynomial
	$$\scalemath{0.76}{P(\lambda)=\begin{bmatrix}2.7694 & 0.7254 & -0.2050\\-1.3499 & -0.0631 & -0.1241\\3.0349&.7147&1.4897\end{bmatrix}+
		\lambda\begin{bmatrix}1.4090&-1.2075&0.4889\\1.4172&0.7172&1.0347\\0.6715&1.6302& 0.7269\end{bmatrix}+
		\lambda^2\begin{bmatrix}-0.3034 & 0.8884 & -0.8095\\ 0.2939&-1.1471&-2.9443\\-0.7873&-1.0689&1.4384\end{bmatrix}}.$$
	\cref{tb21} and \cref{tb22} record the computed values of the distances $\delta_F(P,0,r)$ and $\delta_F(P,-1,r)$ respectively obtained via BFGS and {\tt globalsearch.m} algorithms for all possible values of $r$ together with the upper and lower bounds.
	The corresponding quantities for the distances $\delta_2(P,0,r)$ and $\delta_2(P,-1,r)$ are recorded in \cref{tb23} and \cref{tb24} respectively except that in these cases the computed value of the distance is obtained only via the {\tt globalsearch.m} algorithm.}
	
	\begin{table}[h!]
		\begin{center}
			\begin{tabular}{|c|c|c|c|c|c|}
				\hline
				\rotanti{Distance} \rotanti{measured} & \rotanti{Lower bound} ~\rotanti{(\cref{lb2})}  & \rotanti{Lower bound} ~\rotanti{(\cref{lb1})}   & \rotanti{BFGS} & \rotanti{{\tt globalsearch}} & \rotanti{Upper bound} ~\rotanti{(\cref{ubound})}\\ \hline
				$\delta_F(P,0,2)$ & 0.25800277& 0.25750097    & 0.25904415  & 0.25904415         & 0.268796      \\ \hline
				$\delta_F(P,0,3)$ & 0.43621850& 0.38556596    & 0.69617957  & 0.69617957         & 0.82200773     \\ \hline
				$\delta_F(P,0,4)$ & 0.88752500& 0.83727454    & 1.84231345 &  1.84231345       & 2.04437686     \\ \hline
				$\delta_F(P,0,5)$ & 1.19949290& 1.13421484    & 1.84468801 &  1.84468801       & 2.43953618     \\ \hline
				$\delta_F(P,0,6)$ & 1.28885600& 1.07999296    & 2.60665217 &  2.60665222       & 2.76918876     \\
				\hline
			\end{tabular}
			\caption{\label{tb21}Comparison of upper and lower bounds with the distance $\delta_F(P,0,r)$ calculated by BFGS and {\tt globalsearch.m} for \cref{ex2}.}
		\end{center}
	\end{table}
	
	\begin{table}[h!]
		\begin{center}
			\begin{tabular}{|c|c|c|c|c|c|}
				\hline
				\rotanti{Distance} \rotanti{measured} & \rotanti{Lower bound} ~\rotanti{(\cref{lb2})}  & \rotanti{Lower bound} ~\rotanti{(\cref{lb1})}   & \rotanti{BFGS} & \rotanti{{\tt globalsearch}} & \rotanti{Upper bound} ~\rotanti{(\cref{ubound})}\\ \hline
				$\delta_F(P,-1,2)$& 0.99413714  & 0.49049043 & 1.14436402 &1.14436402  &1.14869786   \\ \hline
				$\delta_F(P,-1,3)$& 1.23816383  & 0.57712979 & 2.22703947 &2.22703947  &2.37565159     \\ \hline
				$\delta_F(P,-1,4)$& 1.33820455  & 0.56416354 & 2.33112163 &2.33112163  &2.51177974     \\ \hline
				$\delta_F(P,-1,5)$& 1.36050277  & 0.59682624 & 2.44152499 &2.44152500  &2.89719526      \\ \hline
				$\delta_F(P,-1,6)$& 1.46702487  & 0.61024547 & 2.62503371 &2.64810973  &2.93776340     \\
				\hline
			\end{tabular}
			\caption{\label{tb22}Comparison of upper and lower bounds with the distance $\delta_F(P,-1,r)$ calculated by BFGS and {\tt globalsearch.m} for \cref{ex2}.}
		\end{center}
	\end{table}
	
	\begin{table}[h!]
		\begin{center}
			\begin{tabular}{|c|c|c|c|c|c|}
				\hline
				\rotanti{Distance} \rotanti{measured} & \rotanti{Lower bound} ~\rotanti{(\cref{lb2})}  & \rotanti{Lower bound} ~\rotanti{(\cref{lb1})} & \rotanti{{\tt globalsearch}} & \rotanti{Upper bound} ~\rotanti{(\cref{ubound})}\\ \hline
				$\delta_2(P,0,2)$ & 0.25800277& 0.25750097    & 0.25802766  & 0.2581792          \\ \hline
				$\delta_2(P,0,3)$ & 0.43621850& 0.38556596    & 0.47215137  & 0.58937606  \\ \hline
				$\delta_2(P,0,4)$ & 0.88752500& 0.83727454    &1.11581440 & 1.57310992    \\ \hline
				$\delta_2(P,0,5)$ & 1.19949290& 1.13421484    &1.49604879 & 1.83989133     \\ \hline
				$\delta_2(P,0,6)$ & 1.28885600& 1.07999296    &1.90820166 & 2.39309442    \\
				\hline
			\end{tabular}
			\caption{\label{tb23}Comparison of upper and lower bounds with the distance $\delta_2(P,0,r)$ calculated by {\tt globalsearch.m} for \cref{ex2}.}
		\end{center}
	\end{table}
	
	\begin{table}[h!]
		\begin{center}
			\begin{tabular}{|c|c|c|c|c|c|}
				\hline
				\rotanti{Distance} \rotanti{measured} & \rotanti{Lower bound} ~\rotanti{(\cref{lb2})}  & \rotanti{Lower bound} ~\rotanti{(\cref{lb1})}   & \rotanti{{\tt globalsearch}} & \rotanti{Upper bound} ~\rotanti{(\cref{ubound})}\\ \hline
				$\delta_2(P,-1,2)$& 0.99413714  & 0.49049043 &0.99413892  & 1.08915666   \\ \hline
				$\delta_2(P,-1,3)$& 1.23816383  & 0.57712979 &1.44794214 & 1.95311420    \\ \hline
				$\delta_2(P,-1,4)$& 1.33820455  & 0.56416354 &1.49553573 & 1.92278887    \\ \hline
				$\delta_2(P,-1,5)$& 1.36050277  & 0.59682624 &1.70157792  & 2.04844570      \\ \hline
				$\delta_2(P,-1,6)$& 1.46702487  & 0.61024547 &2.19715515 & 2.64000204     \\
				\hline
			\end{tabular}
			\caption{\label{tb24}Comparison of upper and lower bounds with the distance $\delta_2(P,-1,r)$ calculated by {\tt globalsearch.m} for \cref{ex2}.}
		\end{center}
	\end{table}
\end{example}

In almost every case the lower bound from \cref{lb2} is better than the lower bound from \cref{lb1}. The perturbations $\Delta P(\lambda)$ constructed to find the upper bound in \cref{ubound} may also be obtained by using nonzero singular values of $T_\gamma(P,\lambda_0)$ other than $f(\gamma)$ and a corresponding pair of left and right singular vectors. However the resulting upper bound obtained by taking the infimum of $\normp{\Delta P}_s$, $s=2$ or $F$ over all permissible $\gamma$ does not seem to be an improvement over the one already obtained. For instance in \cref{ex1}, the matrix $T_\gamma(P,0)$ corresponding to the distance $\delta_2(P,0,4)$ is of size $8$ and the upper bound from \cref{ubound} reported in \cref{tb13} is constructed by using $\sigma_5(T_\gamma(P,0))$ and it corresponding left and right singular vectors. If the same bound is constructed by considering the three smallest singular values $\sigma_{6}(T_\gamma(P,0)),$ $\sigma_{7}(T_\gamma(P,0))$ and $\sigma_{8}(T_\gamma(P,0))$ and corresponding left and right singular vectors, then the values are $1.55784600,$ $1.65319413$ and $2.42096365$ respectively. Similar observations have been made by considering the other singular value of $T_\gamma(P,0).$

\section{Conclusion} Given a square matrix polynomial $P(\lambda),$ the problem of finding the distance to a nearest matrix polynomial with an elementary divisor of the form $(\lambda- \lambda_0)^j, j \geqslant r,$ for a given $\lambda_0 \in \C$ and $r \geqslant 2$ has been considered. The distance is shown to be zero for singular matrix polynomials. The problem has been characterized in terms of different optimization problems. One of them shows that the solution is the reciprocal of a generalized notion of a $\mu$-value.  The other optimization is used to compute the distance via numerical software like BFGS and MATLAB's {\tt globalsearch}. Upper and lower bounds have been derived from the characterizations and numerical experiments performed to compare them with the computed values of the distance show that they are quite tight in many cases. Since $\mu$-value computation is an NP-hard problem, it is conjectured that the solution of the given distance problem is also NP-hard. The optimizations involved in the calculations are computationally quite expensive. But this is also the case with other optimizations proposed in the literature for computing similar distances. Also due to the nature of the optimizations, it is not clear that the values of the bounds from \cref{lb2} and \cref{ubound} are the globally optimal values. However in many cases they are very close to the computed values of the distance. This leaves the question whether they may actually give the exact solution of the distance problem open for future research.

\bibliographystyle{plain}
\bibliography{Nearest_polynomials_specified_Jordan_chains}
	
\end{document}